\documentclass[10pt]{article}
\usepackage{amssymb}
\usepackage{amsthm}
\usepackage{graphicx}
\newtheorem{theorem}{Theorem}
\newtheorem{lemma}{Lemma}
\newtheorem{corollary}{Corollary}
\newtheorem{remark}{Remark}

\def\Frac#1#2{\frac{\displaystyle{#1}}{\displaystyle{#2}}}
\def\cali{{\cal{I}}}
\begin{document}
 \title{Simple bounds with best possible accuracy for ratios of modified Bessel functions}

\author{Javier Segura\\
        Departamento de Matem\'aticas, Estad\'{\i}stica y 
        Computaci\'on,\\
        Universidad de Cantabria, 39005 Santander, Spain.\\
        javier.segura@unican.es
}

\date{ }

\maketitle
\begin{abstract}
The best bounds of the form 
$B(\alpha,\beta,\gamma,x)=(\alpha+\sqrt{\beta^2+\gamma^2 x^2})/x$ for ratios of modified Bessel functions are characterized:
if $\alpha$, 
$\beta$ and $\gamma$ are chosen in such a way 
that $B(\alpha,\beta,\gamma,x)$ is a sharp approximation for $\Phi_{\nu}(x)=I_{\nu-1} (x)/I_{\nu}(x)$ 
as $x\rightarrow 0^+$
 (respectively $x\rightarrow +\infty$) and the graphs of the functions $B(\alpha,\beta,\gamma,x)$ and 
$\Phi_{\nu}(x)$ are tangent at some $x=x_*>0$, then $B(\alpha,\beta,\gamma,x)$ is an upper (respectively lower) 
bound for $\Phi_{\nu}(x)$ for any positive $x$, and it is the best possible at $x_*$. 
The same is true for the ratio $\Phi_{\nu}(x)=K_{\nu+1} (x)/K_{\nu}(x)$ 
but interchanging lower and upper bounds (and with 
a slightly more restricted range for $\nu$). 
Bounds with maximal accuracy at $0^+$ and $+\infty$ are 
recovered in the limits $x_*\rightarrow 0^+$ and $x_*\rightarrow +\infty$, and for these cases the coefficients 
have simple expressions. For the case of finite and positive $x_*$ 
we provide uniparametric families of bounds which are close to the optimal bounds and retain their
confluence properties. 

\end{abstract}

{\bf Keywords:} Modified Bessel functions, ratios, best bounds.

{\bf MSC2020:.} 33C10, 26D07, 41A99.

\section{Introduction}

Modified Bessel functions and in particular the ratios of the first and second kind modified 
Bessel functions $I_{\nu-1}(x)/I_{\nu}(x)$ and 
$K_{\nu+1}(x)/K_{\nu}(x)$ are mathematical functions appearing in a huge number of applications.
In many of these applications, sharp bounds for estimating these ratios are important; for recent 
references (later than 2019) where these bounds play an 
important role see for instance \cite{ap6,ap1,ap7,ap8,ap3,ap9} for applications where 
bounds for the ratios of first kind Bessel functions appear and 
\cite{ap1,ap7,ap5,ap2,ap4,ap10} for the case of second kind Bessel functions.
The ratios of modified Bessel functions
are important functions on their own and it is no surprise that bounding these ratios has been
a topic of interest for many authors; see for instance 
\cite{Amos:1974:COM,Baricz:2009:OAP,Hornik:2013:ABF,Laforgia:2010:SIF,
Ruiz:2016:ANT,Segura:2011:BFR,Segura:2021:COM,Simpson:1984:SMR,Yang:2018:MOF,Yuan:2000:OTB}. 

In most of the aforementioned papers (with the exception of \cite{Ruiz:2016:ANT,Segura:2021:COM}) 
the bounds are of
the form 
\begin{equation}
\label{type}
B(\alpha,\beta,\gamma,x)=\Frac{\alpha+\sqrt{\beta^2+\gamma^2 x^2}}{x}.
\end{equation}
These bounds 
are widely used because they can be quite sharp, they are simple and it is easy to operate
with them; other type of bounds (like in \cite{Ruiz:2016:ANT,Segura:2021:COM}) can be sharper but are not so easy
to handle.
The present paper culminates these previous studies on algebraic bounds of the form 
$B(\alpha,\beta,\gamma,x)$ 
in two senses: firstly, we summarize  the previous algebraic bounds and classify and collect them 
in a table according to the accuracies at $x=0,+\infty$. Such
classification will reveal that the set of most accurate bounds at $x=0$ and/or $x=+\infty$ is complete
for the ratio $I_{\nu-1}(x)/I_{\nu}(x)$, with bounds earlier described in 
\cite{Amos:1974:COM,Simpson:1984:SMR,Yang:2018:MOF},
but not so for $K_{\nu+1}(x)/K_{\nu}(x)$; this second case 
is completed in this paper. 
In the second place, we will recover the idea considered in 
\cite{Hornik:2013:ABF} of building the best
possible (and close to best possible) bounds around any given fixed positive value of $x$, but we go beyond
the particular case of best lower bounds for first kind Bessel functions and we discuss 
both lower and upper bounds, and for both the first and second kind functions. 
 
 The main results we will obtain can be stated in a simple way (not counting details on the range of validity with respect to $\nu$, which we later describe carefully). Namely, if $\alpha$, 
$\beta$ and $\gamma$ are chosen
such that $B(\alpha,\beta,\gamma,x)$ is a sharp approximation for $\Phi_{\nu}(x)=I_{\nu-1} (x)/I_{\nu}(x)$ 
as $x\rightarrow 0^+$
 (respectively $x\rightarrow +\infty$) and the graphs of the functions $B(\alpha,\beta,\gamma,x)$ and 
$\Phi_{\nu}(x)$ are tangent at some $x=x_*>0$, then $B(\alpha,\beta,\gamma,x)$ is an upper (respectively lower) 
bound for $\Phi_{\nu}(x)$; the same is true for the ratio $\Phi_{\nu}(x)=K_{\nu+1} (x)/K_{\nu}(x)$ 
but interchanging lower and upper bounds.

This analysis will 
 complete the description of the best possible upper and lower bounds of the form $B(\alpha,\beta,\gamma,x)$ around any given $x_*\in [0,+\infty]$. It will not be possible to give simple expressions for the coefficients
$\alpha$, $\beta$ and $\gamma$ as a function of $x_*$ when $x_*\in (0,+\infty)$, but we will see how it is possible
 to give explicit 
bounds which are close to the best bounds (in fact, we will first derive the close to best bounds and later 
the best bounds). Bounds with maximal accuracy at $0$ and $+\infty$ are 
recovered in the limits $x_*\rightarrow 0^+$ and $x_*\rightarrow +\infty$, both for the best and close
 to best bounds, and these limiting cases are
contained
in the previously mentioned table of bounds. 

The structure of the paper is as follows. We start by summarizing the existing bounds for the ratios of modified Bessel functions in section \ref{riccati}, and in section \ref{class} 
we classify them in a systematic way according to their accuracy 
at $x=0$ and $x=+\infty$, filling a gap in this classification for 
the second kind functions. This classification, on the other hand, will reveal a
 mirror symmetry between the
bounds for the first and second kind functions. Additionally, we will discuss two additional bounds with enhanced accuracy which result  from modifying slightly the expression $B(\alpha,\beta,\gamma,x)$.
In section \ref{closeto} we build uniparametric families of bounds which
 are close to the best algebraic bounds of the form $B(\alpha,\beta,\gamma,x)$, 
 and which also exhibit the mirror symmetry. 
These close to best bounds are the starting point for 
proving the existence of the best bounds in section \ref{bestbo}. For the best bounds, 
properties of the 
coefficients $\alpha$, $\beta$ and $\gamma$ as functions of the tangency point $x_*$
are also discussed.

\section{Review of previous results}
\label{riccati}

Most of the sharp bounds for ratios of modified Bessel function available so far 
can be obtained using very similar ideas, starting from the difference-differential system 
\cite[10.29.2]{Olver:2010:BF}
\begin{equation}
\label{DDE}
\begin{array}{l}
\cali'_{\nu}(x)=\cali_{\nu-1}(x)-\Frac{\nu}{x}\cali_{\nu}(x),\\
\\
\cali'_{\nu-1}(x)=\cali_{\nu}(x)+\Frac{\nu -1}{x}\cali_{\nu-1}(x)
\end{array}
\end{equation}
(where $\cali_{\nu}(x)$ denotes $I_{\nu}(x)$, $e^{i \pi \nu} K_{\nu}(x)$ or any linear
 combination of them),  together 
with the unique behavior of  $I_{\nu}(x)$ as $x\rightarrow 0^+$ and of
$K_{\nu}(x)$ as $x\rightarrow +\infty$. 

Starting from the DDE (\ref{DDE}) we can obtain the Riccati equation satisfied by 
\begin{equation}
h_{\mu,\nu}(x)=x^{-\mu}\Phi_{\nu}(x)=x^{-\mu}\Frac{\cali_{\nu-1}(x)}{\cali_{\nu}(x)}.
\end{equation}
which is
\begin{equation}
\label{Riccati}
h'_{\mu,\nu}(x)=x^{-\mu}+\Frac{2\nu-\mu-1}{x}h_{\mu,\nu}(x)-x^{\mu}h_{\mu,\nu}(x)^2 .
\end{equation}

As described in \cite{Ruiz:2016:ANT,Segura:2021:COM}, solving  $h'_{\mu,\nu}(x)=0$ for $h_{\mu,\nu}(x)$
(which gives the nullclines of the 
Riccati equation), we get bounds for $h_{\mu,\nu} (x)$ 
for certain values of $\mu$,
and in particular for $\mu=-1,0,1$. Next we summarize such bounds together with 
those that can be extracted
from the application of the recurrence relation
\begin{equation}
\label{TTRR}
\cali_{\nu+1}(x)+\Frac{2\nu}{x}\cali_\nu (x)-\cali_{\nu-1}(x)=0.
\end{equation}

\subsection{Bounds from the Riccati equation and the recurrence relation}

From the analysis of the Riccati equation satisfied by $h_{\mu,\nu}(x)$, the following result 
can be proved (\cite[Thm. 1]{Segura:2021:COM}; see also \cite{Ruiz:2016:ANT}) 
\footnote{In this section the same compact notation for the bounds
as in \cite{Segura:2021:COM} is used, which is later droped in favor of the more general notation
$B(\alpha,\beta,\gamma,x)$ (for all the bounds in this section $\gamma=1$)}
:

\begin{theorem} 
\label{firstbo}
Let
$\lambda_{\mu,\nu}(x)=\Frac{1}{x}\left\{\nu-\Frac{\mu+1}{2}+\sqrt{\left(\nu-\Frac{\mu+1}{2}\right)^2+x^2}\right\},
$
the following bounds hold for real positive $x$:
\begin{enumerate}
\item{}$
I_{\nu-1}(x)/I_{\nu}(x)<\lambda_{-1,\nu}(x),\,\nu\ge 0
$
\item{}$
I_{\nu-1}(x)/I_{\nu}(x)>\lambda_{0,\nu}(x),\,\nu\ge 1/2
$
\item{}$
I_{\nu-1}(x)/I_{\nu}(x)>\lambda_{1,\nu}(x),\,\nu\ge 0
$
\item{}$
K_{\nu}(x)/K_{\nu-1}(x)<\lambda_{-1,\nu}(x),\,\nu\in {\mathbb R}
$
\item{}$
K_{\nu}(x)/K_{\nu-1}(x)<\lambda_{0,\nu}(x),\,\nu >1/2
$
\item{}$
K_{\nu}(x)/K_{\nu-1}(x)>\lambda_{1,\nu}(x),\,\nu\in {\mathbb R}
$
\end{enumerate}
\end{theorem}

\begin{remark}
$K_{\nu}(x)/K_{\nu-1}(x)=\lambda_{0,\nu}(x)=1$ if $\nu =1/2$.
\end{remark}

\begin{remark}
The first three bounds were first proved in \cite{Amos:1974:COM} for $\nu\ge 1$. The validity 
of the first and third bounds was extended to $\nu \ge 0$ in \cite{Yuan:2000:OTB} and the range
for the second bound was extended in \cite{Segura:2011:BFR} to $\nu\ge 1/2$ (see also \cite{Ruiz:2016:ANT}).
The fourth bound was first proved in \cite{Laforgia:2010:SIF}, the fifth in \cite{Segura:2011:BFR} and 
the last one in \cite{Ruiz:2016:ANT}.
\end{remark}

The recurrence relation (\ref{TTRR}) can be used to generate further bounds, as considered in 
\cite{Segura:2011:BFR}. For this
purpose we write the recurrence as
\begin{equation}
\label{fc1}
\Phi_{\nu}(x)=\Frac{2\nu}{x}+\Phi_{\nu+1}(x)^{-1},
\end{equation}
and using an upper (respectively lower) bound for $\Phi_{\nu+1}(x)=\Frac{I_{\nu}(x)}{I_{\nu+1}(x)}$ we obtain
a lower (respectively upper) bound for $\Phi_{\nu}(x)$ if $\nu\ge 0$. In particular, considering the bounds
of the form given in Theorem \ref{firstbo} we have:

\begin{lemma}
\label{cross1}
If $\lambda_{\mu,\nu}(x)$ is a positive upper (respectively lower) bound for $\Frac{I_{\nu-1}(x)}{I_{\nu}(x)}$ when 
$\nu\ge\nu_0$ ($\nu_0\ge 1$) then
$$
\tilde{\lambda}_{\mu,\nu}(x)=\Frac{1}{x}\left(\nu+\Frac{\mu-1}{2}+\sqrt{\left(\nu-\Frac{\mu-1}{2}\right)^2+x^2}\right)
$$
is a lower (respectively upper) bound for $\nu \ge 0$.
\end{lemma}

\begin{remark}
\label{fraci}
Taking the case $\mu=1$, that is, 
starting from the bound 3 of Theorem \ref{firstbo} we obtain the first bound 
of this same theorem; both bounds are then related by the recurrence. 
The cases $\mu=-1,0$ provide two additional bounds, which were already described in \cite{Amos:1974:COM}.
\end{remark}

We can also rewrite the recurrence in the forward direction
\begin{equation}
\label{recfo}
-\Phi_{\nu}(x)=\left(\Frac{2(\nu-1)}{x}-\Phi_{\nu-1}\right)^{-1}
\end{equation}
and using an upper (respectively lower) bound for $-\Phi_{\nu-1}(x)=K_{\nu-1}(x)/K_{\nu-2}(x)$ we obtain
a lower (respectively upper) bound for $-\Phi_{\nu}(x)$. Then we have:

\begin{lemma}
\label{le2}
If $\lambda_{\mu,\nu}(x)$ is a positive upper (respectively lower) bound for 
$\Frac{K_{\nu}(x)}{K_{\nu-1}(x)}$ when 
$\nu\ge\nu_0$ then $\hat{\lambda}_{\mu,\nu}(x)$, where
$$
\hat{\lambda}_{\mu,\nu}(x)=\Frac{1}{x}\left(\nu+\Frac{\mu-1}{2}+\sqrt{\left(\nu-\Frac{\mu+3}{2}\right)^2+x^2}\right),
$$
is a lower (respectively upper) bound for $\nu \ge\nu_0 +1$
\end{lemma}

\begin{remark}
\label{frack}
Taking the case $\mu=-1$, that is, 
starting from the bound 4 of Theorem \ref{firstbo} we obtain the sixth bound 
of this same theorem; both bounds are then related by the recurrence.
The cases $\mu=0,1$ provide two additional bounds, which correspond to the bounds of Eq. (34) 
($\mu=1$) and Eq. (35) ($\mu=0$) of reference \cite{Segura:2011:BFR}.
\end{remark}

\subsection{Other bounds}

We end this section summarizing other bounds that will be important in the classification of the best algebraic 
bounds of the form $B(\alpha,\beta,\gamma,x)$. 

\begin{theorem}
\label{extrab}
The following bounds hold for positive $x$:

1. $I_{\nu-1}(x)/I_\nu (x)>\Frac{1}{x}\left(\nu-\frac12+\displaystyle\sqrt{\nu^2-\frac14+x^2}\right),\,\nu\ge 1/2$

2. $I_{\nu-1}(x)/I_\nu (x)>\Frac{1}{x}\left(\nu+\displaystyle\sqrt{\nu^2+\Frac{\nu}{\nu+1}x^2}\right),\,\nu\ge 0$

3. $I_{\nu-1}(x)/I_\nu (x)<\Frac{1}{x}\left(\nu-2+\displaystyle\sqrt{(\nu+2)^2+\Frac{\nu+2}{\nu+1}x^2}\right),\,\nu\ge 0$

4. $K_{\nu+1}(x)/K_{\nu}(x)<\Frac{1}{x}\left(\nu+\sqrt{\nu^2+x^2 \nu/(\nu-1)}\right),\,\nu>1
$

\end{theorem}

The first bound in the previous theorem appeared in \cite{Simpson:1984:SMR} and its range of validity was extended in 
\cite{Hornik:2013:ABF}. This is the only bound we have described so far that is not a direct consequence
of the Riccati equation and the recurrence relation. However, the method of proof is similar in that
it involves simple arguments regarding the monotonicity. Later, we improve this result and provide and analogous (upper) bound for $K_{\nu+1}(x)/K_{\nu}(x)$.

The second and fourth bounds were proved in \cite{Segura:2011:BFR} as a consequence of the Tur\'an-type inequalities
satisfied by modified Bessel functions; such Tur\'an inequalities were again obtained as a consequence of the analysis of the Riccati equation. The second 
bound appeared earlier in \cite{Baricz:2009:OAP}.

Finally, the third bound is the bound (4.10) of \cite{Yang:2018:MOF} 
(notice that the index $\nu$ 
has to be shifted and an extra factor $x$ appears because they are bounding 
$x I_{\nu}(x)/I_{\nu+1}(x)$ instead of $I_{\nu-1}(x)/I_{\nu}(x)$).
 This bound appeared in \cite{Yang:2018:MOF} as the best of a set of uniparametric bounds; the use of the Frobenius series for first kind Bessel functions was considered in the proof.  We note that this bound is a direct consequence of the second bound in the previous theorem and the application of the recurrence relation. Indeed, 
denoting by $L_{\nu}(x)$ this upper bound for $I_{\nu-1}(x)/I_{\nu}(x)$ and using (\ref{fc1}) we have
$$
\Frac{I_{\nu-1}(x)}{I_\nu (x)}<\Frac{2\nu}{x}+L_{\nu+1}(x)^{-1}=\Frac{1}{x}\left(\nu-2+\displaystyle\sqrt{(\nu+2)^2+\Frac{\nu+2}{\nu+1}x^2}\right),\,\nu \ge 0.
$$

The recurrence relation can also be applied in the forward relation (\ref{recfo}),
 but then, as we will discuss, the bound 
will be weaker. We have in this case
$$
\Frac{I_{\nu-1}(x)}{I_{\nu}(x)}=\left(-\Frac{2(\nu-1)}{x}+\Frac{I_{\nu-2}(x)}{I_{\nu-1}(x)}\right)^{-1}<
\left(-\Frac{2(\nu-1)}{x}+L_{\nu-1}(x)\right)^{-1},
$$
which gives

\begin{equation}
\label{otronew}
\Frac{I_{\nu-1}(x)}{I_{\nu}(x)}<\Frac{1}{x}\left(\nu+\displaystyle\sqrt{\nu^2+\Frac{\nu}{\nu-1}x^2}\right),\,\nu > 1.
\end{equation}

\section{Classification of the bounds}
\label{class}

It will be helpful in clarifying the status of the known bounds so far to classify them in some way. A neat 
way to do this 
is to analyze the sharpness of the bounds at $x=0$ and $x=+\infty$, for this purpose, we compare the expansions
\begin{equation}
\label{expaba}
\begin{array}{l}
B(\alpha,\beta,\gamma,x)=\Frac{\alpha+\beta}{x}+\Frac{\gamma ^2 x}{2\beta}-
\Frac{\gamma^4 x^3}{8 \beta^3}+{\cal O}(x^5),\,x\rightarrow 0\\
B(\alpha,\beta,\gamma,x)=\gamma+\Frac{\alpha}{x}+\Frac{\beta^2}{2\gamma x^2}+{\cal O}(x^{-4}),\,x\rightarrow +\infty,
\end{array}
\end{equation}
(assuming, without loss of generality, that $\beta\ge 0$, $\gamma\ge 0$) 
with the expansions for $I_{\nu-1}(x)/I_\nu (x)$ (\ref{expin}) and (\ref{serin}) (and similarly for the second
kind function). We will say that a bound 
$B(\alpha,\beta,\gamma,x)$
has
accuracy $n\in {\mathbb N}$ at $x=0$ if the first $n$ terms of the expansion of $B(\alpha,\beta,\gamma,x)$ 
are the same as the 
first $n$ terms in (\ref{expin}); we define in an analogous way the accuracy at $x=+\infty$.

We first classify the bounds for the first kind Bessel function, and we will see that with respect to the 
accuracy at $x=0$ and $x=+\infty$, the set of bounds described so far is complete in a sense to be described later.
However, this will not be the case of the second kind Bessel function, and the completion of the set of bounds will lead to finding new bounds. Later, we will enhance these sets of bounds with new bounds which are more accurate for
 intermediate values of $x$, completing the description of the best algebraic bounds.

\subsection{Bounds for $I_{\nu-1}(x)/I_{\nu}(x)$}

The bounds $B(\alpha,\beta,\gamma,x)$ we have described for $I_{\nu}(x)/I_{\nu-1}(x)$ are classified in 
Table \ref{tablei}
according to the accuracy at $x=0$ and $x=+\infty$. We assign to each bound a pair $(n,m)$, where $n$ is the 
accuracy at $x=0$ and $m$ the accuracy at $x=+\infty$.

\begin{table}
\label{tablei}
\begin{tabular}{cccccc}
$(n,m)$ & $\alpha$ & $\beta$ & $\gamma$ & Range & Type\\
 $(0,1)$ & $\nu-1$ & $\nu-1$ & $1$ & $\nu\ge 0$ & L\\
$(2,0)$ & $\nu$ & $\nu$ & $\sqrt{\nu/(\nu+1)}$ & $\nu\ge 0$ & L\\
$(0,2)$ & $\nu -\frac12$ & $\nu -\frac12$ & $1$ & $\nu\ge \frac12$ &  L\\
$(2,1)$ & $\nu -1$ & $\nu+1$ & $1$ & $\nu\ge 0$ &  L\\
$(0,3)$ & $\nu -\frac12$ & $\sqrt{\nu^2-\frac14}$ & $1$ & $\nu\ge \frac12$ &  L\\
$(1,0)$ & $\nu$ & $\nu$ & $\sqrt{\nu/(\nu-1)}$ & $\nu> 1$ & U\\
$(1,1)$ & $\nu$ & $\nu$ & $1$ & $\nu\ge 0$ & U \\
$(1,2)$ & $\nu -\frac12$ & $\nu+\frac12$ & $1$ & $\nu\ge 0$ &  U\\
$(3,0)$ & $\nu -2$ & $\nu+2$ & $\sqrt{(\nu+2)/(\nu+1)}$ & $\nu \ge 0$ &  U\\
\end{tabular}
\caption{Bounds for the ratio $I_{\nu-1}(x)/I_{\nu}(x)$ of the type $B(\alpha,\beta,\gamma,x)=
(\alpha+\sqrt{\beta^2+\gamma^2 x^2})/x$ classified
according to their accuracies at $x=0$ ($n$) and $x=+\infty$ ($m$). The range of validity of the bounds is given, and the type of bound is labeled as L for the lower bounds and U for the upper bounds.}
\end{table}

All the bounds of this table have been described earlier in this paper. The bounds $(0,1)$, $(0,2)$ 
and $(1,1)$ are given 
in Theorem \ref{firstbo}; the bounds $(2,1)$ and $(1,2)$ are described in Lemma \ref{cross1} and Remark \ref{fraci}; 
$(0,3)$, $(2,0)$ and $(3,0)$ are collected in Theorem \ref{extrab} and finally $(1,0)$ is Eq. (\ref{otronew}).

\begin{remark}
\label{propsta}
Some important observations related to this table are: 
\begin{enumerate}
\item{There} may exist more bounds $B(\alpha,\beta.\gamma,x)$ 
with accuracies such that $n+m<3$ apart from those given in the table. For example, considering (\ref{fc1})
we see that $B(2\nu,0,0,x)$ is a lower bound, and this is a bound of the type $(1,0)$.
\item{} Bounds $B(\alpha,\beta,\gamma,x)$ with $n+m>3$ do not exist because the bounds depend of three parameters, and then it is not possible to reproduce
more than three terms in the corresponding expansions. 
\item{The} bounds with $n+m=3$ are unique, and they are all contained in the table, as we next prove in 
Theorem \ref{newi} 
\end{enumerate}

\end{remark}

\begin{theorem}
\label{newi}
Let 
$$
B(\alpha,\beta,\gamma,x)=\Frac{\alpha+\sqrt{\beta^2+\gamma ^2 x^2}}{x}.
$$

$B(\alpha,\beta,\gamma,x)$ is a bound for $I_{\nu-1}(x)/I_{\nu}(x)$ for all $x>0$ and $\nu\ge \nu_0$, with $\nu_0$ not larger
than $1/2$, 
for each of the selections of $\alpha,\,\beta,\,\gamma$ such that three terms of the 
development of $B(\alpha,\beta,\gamma,x)$ in power series as 
$x\rightarrow 0^+$ and/or $x\rightarrow +\infty$ coincide with the corresponding
expansions for $I_{\nu-1}(x)/I_{\nu}(x)$ (three terms in total for both expansions). These correspond to 
the bounds with accuracies $(n,m)$, $n+m=3$.
\end{theorem}
\begin{proof}
Comparing the expansions for $I_{\nu-1}(x)/I_\nu (x)$ (\ref{expin}) and (\ref{serin}) with (\ref{expaba}) 
we see that 
the conditions for the first three terms to coincide with the expansion (\ref{serin}) are 
(we take $\beta\ge 0$ and $\gamma\ge 0$):
\begin{equation}
\label{cond1}
\alpha+\beta=2\nu,\,\beta/\gamma^2=\nu+1,\,\beta^3/\gamma^4=(\nu+1)^2 (\nu+2)
\end{equation}
(in increasing order) while the three conditions for the coincidence as $x\rightarrow +\infty$ are 
(starting from the first term)
\begin{equation}
\label{cond2}
\gamma=1,\,\alpha=\nu-1/2,\,\beta^2/\gamma=\nu^2-1/4.
\end{equation}

Considering, for instance, that the three conditions of (\ref{cond1}) are satisfied we obtain
$$
\alpha=\nu-2,\,\beta=\nu+2,\gamma=\displaystyle\sqrt{\Frac{\nu+2}{\nu+1}},
$$
which is the bound $(3,0)$ of the table. Similarly for $(2,1)$, $(1,2)$ and $(0,3)$, the coefficients
$\alpha$, $\beta$ and $\gamma$ are univocally determined and are as shown in Table \ref{tablei}.

\end{proof}

Let us denote by $B_\nu^{(n,m)}(x)$ the bounds with accuracies $(n,m)$ of Table \ref{tablei}, then 

\begin{lemma}
\label{lemacom}
A bound $B_\nu^{(n_1,m_1)}(x)$ is sharper for all $x>0$ than a bound $B_\nu^{(n_2,m_2)}(x)$ of 
the same type (upper or lower) 
if and only if the following conditions
are met: $n_1\ge n_2$, $m_1\ge m_2$, $n_1+m_1>n_2+m_2$. 
\end{lemma}

\begin{proof}
Assume that  $n_1\ge n_2$, $m_1\ge m_2$, $n_1+m_1>n_2+m_2$. First, it is easy to check that under these conditions 
the bounds in Table \ref{tablei} are such that that $B_\nu^{(n_1,m_1)}(x)\neq B_\nu^{(n_2,m_2)}(x)$ for all $x>0$. With this
it is obvious that $B_\nu^{(n_1,m_1)}(x)$ is necessarily sharper than $B_\nu^{(n_2,m_2)}(x)$. Indeed, the first 
bound is sharper at least at one of the end points $x=0$ o $x=+\infty$ and because 
$B_\nu^{(n_1,m_1)}(x)\neq B_\nu^{(n_2,m_2)}(x)$ for all $x>0$ then it is sharper for all $x$.
\end{proof}

A consequence of Theorem \ref{newi} and Lemma 
\ref{lemacom} is that the bounds with $n+m=3$ are the best
possible of the form $B(\alpha,\beta,\gamma,x)$ at $x=0$ or $x=+\infty$. The set of bounds for the ratio
$I_{\nu-1}(x)/I_{\nu}(x)$ is complete in the sense that all four best bounds at $x=0$ or $x=+\infty$ are given. The next results gives the region when each of these bounds is the best.

\begin{corollary}
\label{corocom}
For $\nu\ge 1/2$ the sharpest lower bound of Table \ref{tablei} is either $B_\nu^{(2,1)}(x)=B(\nu-1,\nu+1,1,x)$ or 
$B_\nu^{(0,3)}(x)=B(\nu-1/2,\sqrt{\nu^2-\frac14},1,x)$, depending on the value of $x$. 
$B_\nu^{(2,1)}(x)$ is the sharpest bound for $x<x_l$, 
$x_l=\sqrt{3(\nu+1/2)(\nu+5/6)}$ and $B_\nu^{(0,3)}(x)$ for $x>x_l$. $B_\nu^{(2,1)}(x)$ is also valid for $\nu\in [0,1/2)$.

For $\nu\ge 0$ the sharpest upper bound is either $B_\nu^{(1,2)}(x)=B(\nu-1/2,\nu+1/2,1,x)$ or 
$B_\nu^{(3,0)}(x)=B(\nu-2,\nu+2,\sqrt{(\nu+2)/(\nu+1)},x)$. $B_\nu^{(3,0)}(x)$ is the sharpest bound for $x<x_u$, 
$x_u=\sqrt{3(\nu+1)(\nu+2)}$ and $B^{(1,2)}(x)$ for $x>x_u$.
\end{corollary}

We stress, as commented earlier in the introduction, that  the best
possible bounds of the form $B(\alpha,\beta,\gamma,x)$ at $x=0$ or $x=+\infty$ for the ratio $I_{\nu-1}(x)/I_{\nu}(x)$ 
were already known, with the bounds $B_\nu^{(1,2)}(x)$ and
$B_\nu^{(2,1)}(x)$ first described in \cite{Amos:1974:COM}, the bound $B_\nu^{(0,3)}(x)$ in \cite{Simpson:1984:SMR}
 (the range validity was extended in \cite{Hornik:2013:ABF}) and the bound $B_\nu^{(3,0)}(x)$ 
in \cite{Yang:2018:MOF}.

\subsection{Bounds for $K_{\nu+1}(x)/K_{\nu}(x)$}

We will denote by $\hat{B}_\nu^{(n,m)}(x)$ the bounds with $n$ correct terms in the expansion as $x\rightarrow 0$
and $m$ correct terms as $x\rightarrow +\infty$ for the ratio $K_{\nu+1}(x)/K_{\nu}(x)$. 
Taking into account the bounds described so far, we notice that the table analogous to 
Table \ref{tablei} is not complete and that, in particular, the bounds $\hat{B}_\nu^{(3,0)}(x)$ and 
$\hat{B}_\nu^{(0,3)}(x)$ are missing. That these bounds of the form $B(\alpha,\beta,\gamma,x)$ exist is guaranteed by
 the next theorem,
which is analogous to Theorem \ref{newi}:

\begin{theorem}
\label{newk}
$B(\alpha,\beta,\gamma,x)$ is a bound for $K_{\nu+1}(x)/K_{\nu}(x)$ for all $x>0$ and $\nu\ge \nu_0$, 
with $\nu_0$ not smaller
than $2$, 
for any of the selections of $\alpha,\,\beta,\,\gamma$ such that the three terms of the 
development of $B(\alpha,\beta,\gamma,x)$ in power series as 
$x\rightarrow 0^+$ and/or $x\rightarrow +\infty$ coincide with the corresponding
expansions for $K_{\nu+1}(x)/K_{\nu}(x)$ (three terms in total for both expansions). 
\end{theorem}

\begin{proof}

Using (\ref{expaba}) and comparing with (\ref{serk1}) and (\ref{serin}) (see the Appendix), the conditions for the 
coincidence of the first three terms of the expansion for $K_{\nu+1}(x)/K_\nu (x)$ at $x=0$ are
\begin{equation}
\label{cond3}
\alpha+\beta=2\nu,\,\beta/\gamma^2=\nu-1,\,\beta^3/\gamma^4=(\nu-1)^2 (\nu-2).
\end{equation}
We observe that the first condition only makes sense if $\nu>0$ because for smaller $\nu$ the first term
in the expansion for $K_{\nu+1}(x)/K_\nu (x)$ is no longer ${\cal O}(x^{-1})$ (see the 
Appendix, Eq. (\ref{serk1})). Similarly,
 the second condition is meaningful only for $\nu>1$ and the third condition for $\nu>2$.

Regarding the conditions as $x\rightarrow +\infty$ we have, considering (\ref{expink}):
\begin{equation}
\label{cond4}
\gamma=1,\,\alpha=\nu+\frac12,\,\beta^2 /\gamma=\nu^2-1/4.
\end{equation}

Let us observe the symmetry between these conditions and the analogous conditions for the bounds of 
$I_{\nu-1}(x)/I_{\nu}(x)$, which will imply that the bounds have very similar expressions.

Let us now consider the $4$ different cases:

$\hat{B}_\nu^{(3,0)}(x)$: we solve the system formed by the three equations in (\ref{cond3}), and we get:
$$
\alpha=\nu+2,\,\beta=\nu-2,\, \gamma=\displaystyle\sqrt{\Frac{\nu-2}{\nu-1}}.
$$
This gives 
$$
\hat{B}_\nu^{(3,0)} (x)=
\Frac{1}{x}\left(\nu+2+\displaystyle
\sqrt{(\nu-2)^2+\Frac{\nu-2}{\nu-1}x^2}
\right),
$$
which is an upper bound for $\nu\ge 2$. Indeed, we consider the fourth bound of Theorem \ref{extrab}
and use the recurrence in the form
\begin{equation}
\label{backk}
\Frac{K_{\nu+1}(x)}{K_\nu (x)}=\Frac{2\nu}{x}+\Frac{K_{\nu-1}(x)}{K_{\nu} (x)},
\end{equation}
yielding
$$
\Frac{K_{\nu+1}(x)}{K_\nu (x)}
>\Frac{2\nu}{x}+\Frac{x}{\nu-1+\sqrt{(\nu-1)^2+\Frac{\nu-1}{\nu-2}x^2}}=\hat{B}_\nu^{(3,0)}(x)
$$

$\hat{B}_\nu^{(2,1)} (x)$: $\alpha=\nu+1$, $\beta=\nu -1$, $\gamma=1$
which is the case $\mu=1$ of Lemma \ref{le2}.

$\hat{B}_\nu^{(1,2)} (x)$: $\alpha=\nu+1/2$,$\beta=\nu-1/2$, $\gamma=1$, case $\mu=0$ of Lemma \ref{le2}.

$\hat{B}_\nu^{(0,3)} (x)$: $\alpha=\nu+\frac12$, $\beta=\sqrt{\nu^2-1/4}$, $\gamma=1$, 
and one can prove
that
$$
\Frac{K_{\nu+1}(x)}{K_\nu (x)}<\hat{B}^{(0,3)}_\nu (x)=\Frac{\nu+1/2+\sqrt{\nu^2+x^2-\frac14}}{x},\,\nu> 1/2.
$$
In fact, we are giving in Theorem \ref{gapk} a sharper bound.
\end{proof}

\begin{table}
\label{tablek}
\begin{tabular}{cccccc}
$(n,m)$ & $\alpha$ & $\beta$ & $\gamma$ & Range & Type\\
$(0,1)$ & $\nu+1$ & $\nu+1$ & $1$ & $\nu\in {\mathbb R}$ & U\\
$(2,0)$ & $\nu$ & $\nu$ & $\sqrt{\nu/(\nu-1)}$ & $\nu>1$ & U\\
$(0,2)$ & $\nu+\frac12$ & $\nu+\frac12$ & $1$ & $\nu\ge -1/2$ & U\\
$(2,1)$ & $\nu+1$ & $\nu-1$ & $1$ & $\nu\in {\mathbb R}$ & U\\
$(0,3)$ & $\nu+\frac12$ & $\sqrt{\nu^2-\frac14}$ & $1$ & $\nu >1/2$ & U\\
$(1,0)$ & $\nu$ & $\nu$ & $\sqrt{\nu/(\nu+1)}$ & $\nu\ge 0$ & L\\
$(1,1)$ & $\nu$ & $\nu$& $1$ & $\nu\in {\mathbb R}$ & L\\
$(1,2)$ & $\nu+\frac12$ & $\nu-\frac12$ & $1$ & $\nu>1/2$ & L\\
$(3,0)$ & $\nu +2$ & $\nu- 2$ & $\sqrt{(\nu-2)/(\nu-1)}$ & $\nu \ge 2$ & L\\
\end{tabular}
\caption{Bounds for the ratio $K_{\nu+1}(x)/K_{\nu} (x)$. For the $(0.3)$ and $(1,2)$ bounds the
equality holds when $\nu=1/2$.}
\end{table}

The bounds
appearing in table \ref{tablek} either have already been described in the paper or a direct consequence 
of the application of the recurrence. Firstly, the $(0,1)$, $(0,2)$ and $(1,1)$ bounds are the results 4, 5 and 6
of Theorem  \ref{firstbo}; the cases $(0,1)$ and $(1,1)$ are related by the recurrence
(case $\mu=-1$ of Lemma \ref{le2}). Similarly, the case $(1,2)$ is connected with $(0,2)$ and the case 
$(2,1)$ with $(1,1)$ (cases $\mu=0,1$ of Lemma \ref{le2}). The case $(2,0)$ is the result 4 of Theorem \ref{extrab}
 and the case $(3,0)$, as shown in the proof of Theorem \ref{newk}, is related to $(2,0)$ through the recurrence 
 relation; similarly, $(1,0)$ can be obtained from $(2,0)$ but applying the recurrence if the opposite direction.
 Finally, the  $(0,3)$ case is proved in Theorem \ref{gapk}, 
 which in fact gives and improvement of this $(0,3)$ bound.

\begin{remark}
\label{simeta}
We note the clear symmetry between Tables \ref{tablei} and \ref{tablek}. If we take an upper (or lower) 
bound for $I_{\nu-1}(x)/I_{\nu}(x)$ and when $\nu+\mu$ appears we replace this value 
by $\nu-\mu$ then we have a lower (respectively upper) bound for $K_{\nu+1}(x)/K_{\nu} (x)$. 
\end{remark}

\begin{remark}
\label{compa2}
Considering the $(1,1)$ bounds both in \ref{tablei} and \ref{tablek} we conclude that 
$K_{\nu+1}(x)/K_{\nu}(x)>I_{\nu-1}(x)/I_{\nu}(x)$, $\nu\ge 0$. Therefore, the lower bounds for
$I_{\nu-1}(x)/I_{\nu}(x)$ are also lower bounds for $K_{\nu+1}(x)/K_{\nu}(x)$ (though not as sharp)
and the upper bounds for $K_{\nu+1}(x)/K_{\nu}(x)$ 
are also upper bounds for $I_{\nu-1}(x)/I_{\nu}(x)$ (but, again, not as sharp).
\end{remark}

 Remark \ref{propsta} also holds for the results of Table \ref{tablek}. Lemma \ref{lemacom} holds 
 for the bounds $\hat{B}^{(n,m)}_\nu (x)$ of table \ref{tablek} too, while 
Corollary \ref{corocom} holds with the appropriate replacements (see Remark \ref{simeta}) and minor modifications. 
Next we give this last 
result
explicitly:

\begin{corollary}
\label{corocom}
For $\nu\ge 1/2$ the sharpest upper bound of Table \ref{tablei} is either $\hat{B}_\nu^{(2,1)}(x)=B(\nu+1,\nu-1,1,x)$ or 
$\hat{B}_\nu^{(0,3)}(x)=B(\nu+1/2,\sqrt{\nu^2-\frac14},1,x)$, depending on the value of $x$. 
If $\nu>5/6$, $\hat{B}_\nu^{(2,1)}(x)$ is the sharpest bound for $x<x_l$, 
$x_l=\sqrt{3(\nu-1/2)(\nu-5/6)}$ and $\hat{B}_\nu^{(0,3)}(x)$ for $x>x_l$. For $\nu\in [1/2,5/6]$ $\hat{B}_\nu^{(0,3)}(x)$
is sharper for all positive $x$. $\hat{B}_\nu^{(2,1)}(x)$ is valid for all real 
$\nu$.

For $\nu\ge  2$ the sharpest lower bound is either $\hat{B}_\nu^{(1,2)}(x)=B(\nu+1/2,\nu-1/2,1,x)$ or 
$\hat{B}_\nu^{(3,0)}(x)=B(\nu+2,\nu-2,\sqrt{(\nu-2)/(\nu-1)},x)$. $\hat{B}_\nu^{(3,0)}(x)$ is the sharpest bound for $x<x_u$, 
$x_u=\sqrt{3(\nu-1)(\nu-2)}$ and $\hat{B}_\nu^{(1,2)}(x)$ for $x>x_u$. $\hat{B}_\nu^{(1,2)}(x)$ is valid for $\nu>1/2$.
\end{corollary}

\subsection{Two additional bounds of the type {\boldmath $(1,3)$}}

We end this section discussing a slightly different type of bound to the rest of the
paper. These two bounds are not of the form (\ref{type}), but are a minor modification. 
These new bounds improve the bounds of the type $(0,3)$ in tables \ref{tablei} and  \ref{tablek}.

\begin{theorem} 
\label{gapk}
Let $\phi_{-,\nu}(x)=xI_{\nu-1}(x)/I_{\nu}(x)$ and $\phi_{+,\nu}(x)=xK_{\nu+1}(x)/K_{\nu}(x)$,
then both functions satisfy the following properties for $\nu\ge 1/2$ and $x>0$
$$
0<\phi'_{\pm,\nu}(x)\le 1,
$$
$$
B_{\nu}^{(1,3)}(x)\equiv \nu+\sqrt{\nu^2+x(x-1)}<\phi_{\pm,\nu}(x)\le\nu+\sqrt{\nu^2 + x(x+1)}
\equiv \hat{B}_{\nu}^{(1,3)}(x),
$$
where the equalities only take place for 
$\phi_{+,\nu}(x)$ when $\nu=1/2$.
The upper bound for $\phi_{+,\nu}(x)$ and
the lower bound for $\phi_{-,\nu}(x)$ are of accuracy $(1,3)$.
\end{theorem}

\begin{proof}

The case $\nu=1/2$ is trivial. Let us assume that $\nu>1/2$. 

We have that
\begin{equation}
\label{riph}
x\phi'_{\pm,\nu}(x)=\mp (x^2+2\nu\phi_{\pm,\nu}(x)-\phi_{\pm,\nu} (x)^2).
\end{equation}
We observe that the 
inequalities $(1,1)$ of tables \ref{tablei} and \ref{tablek} together with (\ref{riph}) 
imply that $\phi'_{\pm,\nu}(x)>0$

Next we prove the inequalities for 
$\phi_{\pm,\nu}(x)$, from which the upper bound for the derivatives follows immediately. 
Considering Remark 
\ref{compa2} we have that $\phi_{-,\nu}(x)<\phi_{+,\nu}(x)$ and 
we only need to prove the upper bound 
for $\phi_{+,\nu}(x)$ and 
the lower bound for $\phi_{-,\nu}(x)$.
 
Now let $\delta_{\pm,\nu}(x)=h_{\pm,\nu}(x)-\phi_{\pm,\nu} (x)$, 
where $h_{\pm,\nu}(x)=\nu+\sqrt{\nu^2+x(x\pm 1)}$. 
Taking the derivative and using (\ref{riph}) we can write:
$$
\delta_{\pm,\nu}' (x)=\Frac{x\pm 1/2}{\sqrt{\nu^2+x(x\pm 1)}}\mp\Frac{1}{x}(x^2+2\nu\phi_{\pm,\nu}(x)-\phi_{\pm,\nu}(x)^2)
$$ 
which we can write as
$$
\delta_{\pm,\nu}'(x)=\Frac{x\pm 1/2}{\sqrt{\nu^2+x(x\pm 1)}}- 1\pm \Frac{1}{x}
\delta_{\pm,\nu}(x)(\delta_{\pm,\nu}(x)+
2(h_{\pm,\nu}(x)-\nu)).
$$
Then, if $x_0>0$ is a value such that $\delta_{\pm,\nu}(x_0)=0$,
because $x\pm 1/2<\sqrt{\nu^2+x(x\pm 1)}$ if $\nu>1/2$, we would have that 
$\delta_{\pm,\nu}'(x_0)<0$. But this, as we see next, leads to a contradiction, which means
that such $x_0$ can not exist.

Considering the expansions as $x\rightarrow +\infty$ we have 
\begin{equation}
\label{expanco}
\delta_{\pm,\nu}(x)=\pm \Frac{\nu^2-1/4}{4x^3}+{\cal O}(x^{-4}).
\end{equation}
Now, consider the plus sign (upper bound for the ratio of second kind function). We have that 
$\delta_{+,\nu}(x)>0$ for sufficiently large $x$, but then we must have that  $\delta_{+,\nu}(x)>0$
for all positive $x$ because, otherwise, if $x_0$ was that largest value of $x$ for which 
 $\delta_{+,\nu}(x_0)=0$ then we would have, as we have proved, that $\delta_{+,\nu}(x_0)<0$, in contradiction
 with the fact that $\delta_{+,\nu}(x)>0$ if $x>x_0$. This proves the upper bound for the ratio  of
 second kind functions.
 
 With respect to the lower bound for first kind functions, we see that in the limit $x\rightarrow 0$
$$
\delta_{-,\nu}(x)=-\Frac{x}{2\nu}+{\cal O}(x^2)
$$
and therefore $\delta_{-,\nu}(x)<0$ for $x$ sufficiently close to $0$, and it has to stay the same
for all $x>0$. Otherwise, if we let $x_0>0$ to be 
the smallest positive value of $x$ such that $\delta_{-,\nu}(x_0)=0$, 
we would have $\delta_{-,\nu}'(x_0)<0$, which is in contradiction with the
 fact that $\delta_{-,\nu}(x)<0$ for $0<x<x_0$.

Finally, we observe that (\ref{expanco}) shows that the corresponding expansions 
(lower bound for $\phi_{-,\nu}(x)$ and upper bound for $\phi_{+,\nu}(x)$) 
have degree of exactness $3$ as
$x\rightarrow +\infty$. On the other hand, because $h_{\pm,\nu}(0)=2\nu$, we see that they are also 
sharp as $x\rightarrow 0$. Therefore these are bounds of type $(1,3)$.

\end{proof}

Now, because it is easy to prove that the bound $B^{(1,3)}(x)$ (respectively 
$\hat{B}^{(1,3)}(x)$)
is sharper than the bound $B^{(0,3)}(x)$ (respectively 
$\hat{B}^{(0,3)}(x)$)
in table \ref{tablei} (respectively table \ref{tablek}), we have the following:
\begin{corollary}
For $\nu\ge 1/2$ the following holds:

$L_{\nu}(x)=\max\{B_{\nu}^{(2,1)}(x),B_{\nu}^{(1,3)}(x)\}$ is sharper than any other lower bound 
for $I_{\nu-1}(x)/I_{\nu}(x)$ of the type 
(\ref{type}) for any $x>0$; 
$B_{\nu}^{(2,1)}(x)<B_{\nu}^{(1,3)}(x)$ if and only if $x>\frac43 (1+\nu)$.

$U_{\nu}(x)=\min\{\hat{B}_{\nu}^{(2,1)}(x),\hat{B}_{\nu}^{(1,3)}(x)\}$ 
is sharper than any other upper bound 
for $K_{\nu+1}(x)/K_{\nu}(x)$ of the type 
(\ref{type}) for any $x>0$; 
$\hat{B}_{\nu}^{(2,1)}(x)>\hat{B}_{\nu}^{(1,3)}(x)$ if and only if $x>\frac43 (1-\nu)$.
\end{corollary}

\section{Close to best possible bounds}
\label{closeto}

In the rest of the paper, we refer exclusively to bounds of the type (\ref{type}). The goal is to 
characterize which are the best possible bounds of this type, not necessarily around $x=0$ or 
$x=\infty$.

Restricting to the bounds of type (\ref{type}), 
those with total accuracy $n+m=3$ at $0$ and $+\infty$ are the best possible.
As we will see, 
it is also possible to consider bounds of this form which are
the best possible at any chosen point $x_*>0$, although the coefficients are not explicitly computable
in terms of elementary functions. Before this, we will obtain explicitly computable bounds which are close 
to these best bounds, and that are sharper than the bounds with maximal total accuracy at $0$ and $+\infty$ 
in finite positive intervals.

\subsection{Bounds with interpolation at $x=+\infty$}

In this section we consider bounds of the 
form $B(\alpha_{\nu}(\lambda),\beta_{\nu}(\lambda),1,x)$, which are sharp as $x\rightarrow +\infty$. 
We start with 
the first kind Bessel function, studied in \cite{Hornik:2013:ABF}.

\subsubsection{Close to best lower bound for {\boldmath $I_{\nu-1}(x)/I_{\nu}(x)$}}

We adapt Theorem 7 of \cite{Hornik:2013:ABF} in the following way:

\begin{theorem} 
\label{horn}
The following holds for $\lambda\in [0,1/2]$,$\nu\ge \frac12 -\lambda$ and $x>0$
\begin{equation}
\begin{array}{r}
\label{tang}
\Frac{I_{\nu-1}(x)}{I_{\nu}(x)}> L_\nu^{(I)}(\lambda,x)=B(\alpha^{(I)}_{\nu}(\lambda),\beta^{(I)}_{\nu}(\lambda),1,x).
\end{array}
\end{equation}
where $\alpha_{\nu}^{(I)}(\lambda)=\nu-1/2-\lambda$, $\beta_{\nu}^{(I)} (\lambda)=\sqrt{2\lambda}+\sqrt{\nu^2-(\lambda-\frac12)^2}$.
\end{theorem}

We don't need to prove this result, which is explained in \cite{Hornik:2013:ABF} (for the sake of clarity: 
notice that in \cite{Hornik:2013:ABF}, the
function that is bounded is $I_{\nu+1}(x)/I_{\nu}(x)$ instead of $I_{\nu-1}(x)/I_{\nu}(x)$). Our version contains minor
modifications, but for brevity we prefer not to duplicate the proof. Instead, we prove in Theorem 
\ref{bnk} a similar result for the ratio $K_{\nu+1}(x)/K_{\nu}(x)$.

\begin{remark}
Observe that
$L_\nu^{(I)}(1/2,x)=B_\nu^{(2,1)}(x)$, $\nu\ge 0$ and 
$L_\nu^{(I)}(0,x)=B_\nu^{(0,3)}(x)$, $\nu\ge 1/2$. 
The uniparametric set of bounds goes continuously from the best lower bound at 
$x=0$ ($\lambda=1/2$) to the best lower bound at $x=+\infty$ ($\lambda=0$) and for $\lambda\in (0,1/2)$ it gives
close to best bounds around finite values of $x=x_*>0$. 

For the rest of the close to best bounds we will obtain in this section 
the same will be true in the sense that they will 
connect continuously the corresponding best (upper or lower) bounds of tables \ref{tablei} or 
table \ref{tablek}. We will not insist in stressing these facts.
\end{remark}

\begin{remark}
The validity of Theorem \ref{horn} extends to all $\lambda \ge 0$ for $\nu\ge|\lambda -1/2|$. 
However, for $\lambda>1/2$ the bounds are 
less sharp
than the bounds for $\lambda=1/2$.
\end{remark}

\subsubsection{Close to best upper bound for {\boldmath $K_{\nu+1}(x)/K_{\nu}(x)$}}

We are going to prove next a similar result to Theorem \ref{horn}, which was proved in 
\cite{Hornik:2013:ABF}, but for the second kind
Bessel function; the proof is similar to the proof in \cite[Thm. 7]{Hornik:2013:ABF}. 
This result, as we see next, maintains a similar type of symmetry as described in 
Remark \ref{simeta}.

\begin{theorem} 
\label{bnk}
The following holds for $\lambda\in[0,1/2]$, $\nu\ge 1/2-\lambda$ and $x>0$:
$$
\Frac{K_{\nu+1}(x)}{K_{\nu}(x)}<U_{\nu}^{(K)}(\lambda,x)=B(\alpha_\nu^{(K)}(\lambda),\beta_\nu^{(K)}(\lambda),1,x),
$$
where
$$
\alpha_\nu^{(K)}(\lambda)=\nu+1/2+\lambda,\,\beta_\nu^{(K)}(\lambda)=-\sqrt{2\lambda}+\sqrt{\nu^2-(\lambda -1/2)^2}.
$$

\end{theorem}

\begin{proof}
 
In the proof, we take $\lambda >0$
(the case $\lambda=0$ corresponds to the bound $(0,3)$ of Table \ref{tablek}).

Let $\phi_{\nu}(x)=xK_{\nu+1}(x)/K_\nu (x)$, which satisfies $x\phi'_\nu(x)=-x^2-2\nu\phi_{\nu}(x)+\phi_\nu (x)^2$, 
and let $\delta(x)=h_{\alpha,\beta}(x)-\phi_\nu (x)$, where $h_{\alpha,\beta}(x)=\alpha+\sqrt{\beta^2 +x^2}$.

Taking $\alpha=\alpha_\nu^{(K)}(\lambda)$ we have 
$\alpha>\nu+1/2$ because $\lambda>0$ and
 on account of (\ref{expaba}) and (\ref{expink}) we find that
 $\delta(x)=h_{\alpha,\beta}(x)-\phi_\nu (x)>0$ for large enough positive $x$. 
Next we prove that, for $\alpha=\alpha_{\nu}^{(K)}(\lambda)$ and 
$\beta=\beta_{\nu}^{(K)}(\lambda)$,
$\delta(x)$ is never zero for any $x>0$ and then that it stays positive for all $x>0$.

Using (\ref{riph}) we can write
\begin{equation}
\label{deltapri}
x\delta' (x)=\Frac{Q_{\alpha,\beta}(s)}{s}+(-2\nu+2h_{\alpha,\beta}(x)-\delta(x))\delta (x),
\end{equation}
where 
\begin{equation}
\label{ese}
s=\sqrt{\beta^2+x^2}
\end{equation}
and 
\begin{equation}
\label{qs}
Q(s)=-\beta^2 +(2\nu\alpha -\alpha^2 -\beta^2)s +(1-2\alpha+2\nu)s^2 .
\end{equation}

The discriminant of the equation $Q(s)=0$ can be written
$$
\Delta=(2\nu\alpha-\alpha^2-\beta^2-2\beta\sqrt{2\alpha-2\nu-1})(2\nu\alpha-\alpha^2-\beta^2+2\beta\sqrt{2\alpha-2\nu-1}).
$$
Solving $\Delta=0$ (using the first factor on the previous expression) we find that one of the solutions is 
$$
\beta=-\sqrt{2\alpha-2\nu-1}+\sqrt{(\alpha-1)(2\nu+1-\alpha)},
$$
which taking $\alpha=\alpha_\nu^{(K)} (\lambda)=\nu+1/2+\lambda$ is precisely the value $\beta_\nu^{(K)}(\lambda)$.
With $\alpha=\alpha_\nu^{(K)}(\lambda)$, $\beta=\beta_\nu^{(K)}(\lambda)$ we then have
\begin{equation}
\label{qdo}
Q(s)=-\gamma (s-\sigma)^2,\, \gamma=2\left(\alpha-\nu-\frac12\right),\,\sigma=\Frac{\beta}{\sqrt{\gamma}},
\end{equation}
and $\gamma >0$.

Now, let us assume that $x_0>0$ is the largest value such that $\delta (x_0)=0$. By assuming that such value
 $x_0$ exists we will arrive to a contradiction, which proves that $x_0$ does not exist and therefore 
$\delta (x)$ does not change sign for $x>0$. 

The contradiction occurs because for 
$\alpha=\alpha_{\nu}^{(K)}(\lambda)$ and $\beta=\beta_{\nu}^{(K)}(\lambda)$ it happens that 
$\delta(x_0)=0$ implies that either $\delta'(x_0)<0$ or $\delta'(x_0)=\delta''(x_0)=0$
and $\delta'''(x_0)<0$; but then the sign of the derivatives implies 
that $\delta(x_0+h)<0$ for sufficiently small $h>0$, which is not possible because $\delta(x)>0$ for $x>x_0$ (because 
$\delta(x)>0$ for large enough $x$).

Indeed, because $\delta (x_0)=0$, by (\ref{deltapri}) and $(\ref{qdo})$ we have that
$$
x_0\delta'(x_0)=Q(s_0)/s_0<0,\,s_0=s(x_0)=\sqrt{\beta^2+x_0^2}
$$
if $s_0\neq \sigma$. In the case $s_0=\sigma$ we can prove that $\delta'(x_0)=\delta''(x_0)=0$ and $\delta'''(x_0)<0$. Indeed, using again 
(\ref{deltapri}) and $(\ref{qdo})$ we can write
$$
x\delta'(x)=-\gamma \Frac{(s(x)-\sigma)^2}{s(x)}+\eta(x)\delta(x).
$$
Now, $\delta(x_0)=0$ implies $\delta'(x_0)=0$ because $s(x_0)=\sigma$. Taking a derivative we get that $\delta''(x_0)=0$, and a further
derivation leads to $x_0\delta'''(x_0)=-\gamma (s'(x_0))^2/s(x_0)<0$.

\end{proof}

\begin{remark}
The validity of 
Theorem\ref{bnk}  extends to $\lambda \ge 0$ for $\nu\ge|\lambda -1/2|$. However, for $\lambda>1/2$ the bounds are 
less sharp
than the bounds for $\lambda=1/2$.
\end{remark}

\paragraph{Two auxiliary lemmas}

We now prove two auxiliary lemmas that we will used later for proving the existence of the best algebraic
bound of the type $B(\alpha,\beta,\gamma,x)$.

\begin{lemma}
\label{smayb}
Let $\alpha=\nu+\lambda +1/2$, $\lambda\in (0,1/2)$, 
$\nu>\lambda+1/2$ and $\nu-\lambda-1/2<\beta<\beta_{\nu}^{(K)}(\lambda)$ then
both $s$-roots of $Q(s)=0$, with $Q(s)$ given by (\ref{qs}), are real and greater than $\beta>0$. 
\end{lemma}
\begin{proof}
First we notice that $\beta_{\nu}^{(K)}(\lambda)$ (defined in Theorem \ref{bnk}) 
satisfies $\beta_{\nu}^{(K)}(\lambda)>\nu-\lambda-1/2$ if $\lambda\in (0,1/2)$, $\nu>\lambda+1/2$.
Indeed, $\beta_{\nu}^{(K)}(\lambda)-(\nu-\lambda-1/2)>0$ if $\sqrt{\nu^2+(\lambda -1/2)^2}>
\sqrt{2\lambda}+\nu-\lambda-1/2$, and squaring both sides this is equivalent to
$(1-\sqrt{2\lambda})^2(\nu-\lambda -1/2)>0$, which holds. 

Now we check that both $s$-roots are real. The discriminant of $Q(s)=0$ can be written 
$$
\Delta=\left\{\beta^2-(\beta_{\nu}^{(K)}(\lambda))^2\right\}
\left\{\beta^2-(\beta_{\nu}^{(I)}(\lambda))^2\right\}
$$
where 
$\beta_{\nu}^{(I)}(\lambda)=\sqrt{2\lambda}+\sqrt{\nu^2-(\lambda-1/2)^2}>\beta_{\nu}^{(K)}(\lambda)$. 
Then it is obvious that the discriminant is positive, and therefore $Q(s)=0$ has two real roots. 

Next we prove that the two roots are greater than $\beta>0$. We write $Q(s)=0$ as
\begin{equation}
\label{eqs}
2\lambda s^2 +(\beta^2-\nu^2+(\lambda +1/2)^2)s+\beta^2=0,
\end{equation}
and now we make the change $s=\beta+\sigma$; the resulting equation is
$$
2\lambda \sigma^2 +\left(\beta^2 +4\lambda\beta -\nu^2 +(\lambda+1/2)^2\right)\sigma+
\beta ((\beta+\lambda+1/2)^2-\nu^2)=0.
$$
We know that both $\sigma$-roots are real and we have to prove that they are positive. For this, on account 
of Descartes rule of signs and because we already know that the two roots are real, 
it will be enough to prove that the first and third coefficients are positive 
(which is obvious)
and the second negative. For proving that the second coefficient 
$d(\beta)=\beta^2 +4\lambda\beta -\nu^2 +(\lambda+1/2)^2$ is negative it is enough to prove that 
$d(\beta_{\nu}^{(K)}(\lambda))<0$. This is so because $d(\pm\infty)=+\infty$, $d(0)<0$ and then if
$d(\beta_{\nu}^{(K)}(\lambda))<0$ we will have $d(\beta)<0$ for all 
$\beta \in (0,\beta_{\nu}^{(K)}(\lambda))$.

After some algebra, we find that
$$
d(\beta_{\nu}^{(K)}(\lambda))=-2\sqrt{2\lambda}(1-\sqrt{2\lambda})\Frac{\nu^2-(\lambda+1/2)^2}
{\sqrt{2\lambda}+\sqrt{\nu^2-(\lambda-1/2)^2}},
$$
and then $d(\beta_{\nu}^{(K)}(\lambda))<0$.

\end{proof}

\begin{lemma}
\label{limice}
Under the conditions of the previous lemma and, with $\beta=\beta_{\nu}(\lambda)$ such that $\nu-\lambda-1/2<\beta_{\nu}(\lambda)\le \beta_{\nu}^{(K)}(\lambda)$, 
the roots of $Q(s)=0$, $s_1 (\lambda)$ and $s_2 (\lambda)$, are such
that $\displaystyle\lim_{\lambda\rightarrow \nu-1/2}s_i(\lambda)=0$,
\end{lemma}
\begin{proof}
From Eq. (\ref{eqs}) and setting $\beta=\beta_{\nu}(\lambda)$ 
we see that $s_1 (\lambda)s_2 (\lambda)=\beta_{\nu}(\lambda)^2/(2\lambda)$ and 
$s_1 (\lambda)+s_2 (\lambda)=-(\beta^2-\nu^2+(\lambda +1/2)^2)/(2\lambda)$. Combining both
$$
s_1 (\lambda)s_2 (\lambda)+s_1 (\lambda)+s_2 (\lambda)=(\nu^2-(\lambda +1/2)^2)/(2\lambda)
$$

And from this expression and the fact that both roots are real and positive (as proved in the 
previous lemma) the result is proved.
\end{proof}

\subsection{Bounds with interpolation at $x=0$}

We are now considering that the bounds are sharp as $x\rightarrow 0^+$. For this purpose, we seek bounds of the 
form $B(\alpha,\beta,\gamma,x)$ with $\alpha+\beta=2\nu$ (see Eqs. (\ref{serin}) and (\ref{serk1})), and we will write
$\alpha=\nu-\lambda$, $\beta=\nu+\lambda$. 

We will establish the bounds
for $I_{\nu-1}(x)/I_\nu (x)$ using a similar line of reasoning as for Theorem \ref{bnk}; we will not give 
a explicit proof for the case of $K_{\nu+1}(x)/K_{\nu}(x)$ because the proof is very similar and the symmetry
considerations (Remark \ref{simeta}) will also apply in this case.

\subsubsection{Close to best upper bound for {\boldmath $I_{\nu-1}(x)/I_{\nu}(x)$}}

As a previous step for obtaining the new upper bounds for $xI_{\nu-1}(x)/I_\nu (x)$ we need to obtain a 
convenient expression for the derivative of the difference between this function and the potential bound. 
This is done in the next lemma.

\begin{lemma}
\label{lemaprep}
Let $\delta(x)=h_{\lambda,c}(x)-\phi_\nu (x)$, where 
$h_{\lambda,c}(x)=\nu-\lambda+\sqrt{(\nu+\lambda)^2+c x^2}$ and $\phi_\nu (x)=xI_{\nu-1}(x)/I_\nu (x)$, then
\begin{equation}
\label{deltak}
x\delta'(x)=\Frac{x^2}{s(\nu+\lambda+s)}R(s)+(2h_{\lambda,c}(x)-2\nu+\delta(x))\delta(x)
\end{equation}
where
\begin{equation}
\label{rs}
R(s)=(c-1)s^2+\left[c(\nu-\lambda+1)-\nu-\lambda\right]s+c(\nu+\lambda)
\end{equation}
and $s=\sqrt{(\nu+\lambda)^2+cx^2}$.
\end{lemma}

\begin{proof}
$$
\begin{array}{ll}
x\delta' (x)&=\Frac{cx^2}{s}-x^2-2\nu\phi_\nu (x)+\phi_\nu(x)^2\\
&=Q(s)+(2h_{\lambda,c}(x)-2\nu+\delta(x))\delta(x),
\end{array}
$$
where
$$
Q(s)\equiv \Frac{cx^2}{s}-x^2-2\nu h_{\lambda,c}(x)+h_{\lambda,c}(x)^2 =\Frac{1}{s}\left[T(s)+(c-1)x^2 s\right],
$$
with
$$
\begin{array}{ll}
T(s)&\equiv c(1-2\lambda)x^2-2\lambda (\nu+\lambda)^2+2\lambda (\nu+\lambda) s\\
 & = (1-2\lambda)s^2+2\lambda(\nu+\lambda)s-(\nu+\lambda)^2\\
& =(1-2\lambda)(s-\nu-\lambda)\left(s-\Frac{\nu+\lambda}
{2\lambda-1}\right)\\
& =(1-2\lambda)\Frac{cx^2}{\nu+\lambda+s}\left(s-\Frac{\nu+\lambda}
{2\lambda-1}\right).
\end{array}
$$
And using this last expression for $T$ in the formula for $Q$ we have the result.
\end{proof}

\begin{theorem} 
\label{IU}
The following holds for $\nu\ge 0$, $\lambda\in [1/2,2]$ and $x>0$:
\begin{equation}
\label{iuq}
\Frac{I_{\nu-1}(x)}{I_{\nu}(x)}<U_\nu^{(I)}(\lambda,x)=B(\nu-\lambda,\nu+\lambda,\sqrt{c_\nu^{(I)}(\lambda)},\,x),
\end{equation}
where
$$
c_\nu^{(I)}(\lambda)=\Frac{\nu+\lambda}{\nu-\lambda+2\sqrt{2\lambda}-1}.
$$
\end{theorem}

\begin{proof}

In the proof we assume that $1/2<\lambda<2$. The cases $\lambda=1/2$ and $\lambda=2$ correspond to the bounds 
$B_\nu^{(1,2)}(x)$ and 
$B_\nu^{(3,0)}(x)$ respectively, which are already given in Table \ref{tablei}. Therefore we don't need to consider this limit cases.

We are looking for a bound $h_{\lambda,c}(x)=\nu-\lambda+\sqrt{(\nu+\lambda)^2+c x^2}$ for 
$\phi_{\nu}(x)=xI_{\nu-1}(x)/I_{\nu}(x)$. We start by computing the discriminant $\Delta$ of the equation $R(s)=0$ (see the previous lemma) 
\begin{equation}
\begin{array}{l}
\label{discrimi}
\hspace*{-0.4cm}\Delta=\left[c(\nu-\lambda+1)-\nu-\lambda\right]^2-4 (c-1) c (\nu+\lambda)\\
=\{[\nu+1-(\sqrt{\lambda}-\sqrt{2})^2] c-\nu-\lambda\}\{[\nu+1-(\sqrt{\lambda}+\sqrt{2})^2] c-\nu-\lambda\}
.
\end{array}
\end{equation}

Now consider the equation $\Delta=0$ and 
we solve for $c$ considering the first factor; this gives
$$
c=c_\nu^{(I)}(\lambda)=\Frac{\nu+\lambda}{\nu-\lambda+2\sqrt{2\lambda}-1}.
$$

For this proof, differently to Theorem \ref{bnk} and similarly as for proving the lower bound in
Theorem \ref{gapk}, the condition that will be used is 
the sign of $\delta(0)>0$, and not that of $\delta(+\infty)$. Taking into account (\ref{expaba}) and (\ref{serin}) 
we conclude that for $1/2<\lambda< 2$ we have $\delta(0)>0$: for these parameters the first term in the expansions of 
$I_{\nu-1}(x)/I_{\nu}(x)$ and $U_\nu^{(I)}(\lambda,x)$ at $x=0$ coincide, but the second term in the expansion for the bound is 
larger. Indeed, because we have $\beta=\nu+\lambda$ and $\gamma^2 =c_{\nu}^{(I)}(\lambda)$ (compare with (\ref{expaba})) with 
(\ref{iuq}))
then 
$$
\Frac{\gamma^2}{\beta}=\Frac{c_\nu^{(I)}(\lambda)}{\nu+\lambda}=\Frac{1}{\nu-\lambda+2\sqrt{2\lambda}-1}>\Frac{1}{\nu+1},
$$
because $-\lambda+2\sqrt{2\lambda}-1$ is positive for $\lambda\in (1/2,2)$ 
and it attains its maximum value (which is $1$) at $\lambda=2$.

Because $\delta(0^+)>0$ and $\delta(x)$ is infinitely differentiable for all $x>0$ if $\nu\ge 0$, all that remains to be proved is that if there existed a value $x_0>0$ such that 
$\delta(x_0)=0$ this would mean that either $\delta'(x_0)>0$ or $\delta'(x_0)=\delta''(x_0)=0$ and 
$\delta''' (x_0)>0$, which implies that such $x_0$ does not exist. But with $c=c_\nu^{(I)}(\lambda)$, 
$R(s)$ has a double root and using (\ref{deltak}) we can write
$$
x\delta'(x)=\mu(x)R(s(x))+\eta(x)\delta(x)
$$
with
\begin{equation}
\label{rsx}
R(s(x))=(c-1)(s(x)-\sigma)^2,\sigma=\Frac{\nu+\lambda}{\sqrt{2\lambda}-1}.
\end{equation}
Now, we observe that $c-1>0$ and $\mu(x)=\Frac{x^2}{s(\nu+\lambda+s)}>0$ and we have that $\delta(x_0)$ implies 
$\delta'(x_0)=\mu(x_0)(c-1)(s(x_0)-\sigma)^2>0$ if $s(x_0)\neq \sigma$. On the other hand, if 
$s(x_0)=\sigma$, then $\delta' (x_0)=0$ and differentiating (\ref{rsx}) we obtain $\delta''(x_0)=0$ 
and $x_0\delta'''(x_0)=2\mu (x_0)s'(x_0)^2>0$, which completes the proof.
\end{proof}

\begin{remark}
The validity of the bound of the previous theorem can be extended for non-negative values of $\lambda$ such 
that $c_{\nu}^{(I)}(\lambda)>1$, 
which is guaranteed if $\nu\ge 0$ with $\nu>(\sqrt{\lambda}-\sqrt{2})^2-1$, however the bounds with $\lambda \in [1/2,2]$ 
are sharper than those for values of $\lambda$ outside this interval.
\end{remark}

\paragraph{Two auxiliary lemmas}

We now
 prove two additional lemmas which will be used later in section \ref{bestbo}

\begin{lemma}
\label{signodis}
If $\lambda\in (1/2,2)$ and $\nu\ge 0$ the discriminant $\Delta$ of Eq. (\ref{discrimi}) 
is positive for $0\le c<c_{\nu}^{(I)}(\lambda)$.
\end{lemma}

\begin{proof}
The first factor in (\ref{discrimi}) is negative and the same is true for the second factor. First, if 
$A=\nu+1-(\sqrt{\lambda}+\sqrt{2})^2\ge 0$ we use 
 that $c<c_\nu^{(I)}(\lambda)$ and then
$$
\begin{array}{ll}
[\nu+1-(\sqrt{\lambda}+\sqrt{2})^2] c-\nu-\lambda\le & (\nu+\lambda)\left(\Frac{\nu+1-(\sqrt{\lambda}+\sqrt{2})^2}{\nu+1-(\sqrt{\lambda}-\sqrt{2})^2}-1\right)\\

& = -\Frac{4\sqrt{2}\lambda (\nu+\lambda)}{\nu+1-(\sqrt{\lambda}-\sqrt{2})^2}=-4\sqrt{2}c_\nu^{(I)}(\lambda)
\end{array}
$$
and if $A<0$ then, using that $c\ge 0$
$$
[\nu+1-(\sqrt{\lambda}+\sqrt{2})^2] c-\nu-\lambda\le -\nu-\lambda
$$
\end{proof}

\begin{lemma}
\label{signosol}
If $\lambda\in (1/2,2)$, $\nu\ge 0$ and $\max\{1,\Frac{\nu+\lambda}{\nu+1}\}< c<c_\nu^{(I)}(\lambda)$ 
the two roots of $R(s)=0$ (Eq. (\ref{rs})) 
are such that $s>\nu+\lambda$.
\end{lemma}
\begin{proof}
That both roots are real is known from the previous lemma. 
Now, we substitute $s=\nu+\lambda+\sigma$ in the equation $R(s)=0$ and get $\tilde{R}(\sigma)=0$, with
$$
\tilde{R}(\sigma)=(c-1)\sigma^2+((3\nu+\lambda+1)c-3\nu-3\lambda)\sigma+2(\nu+\lambda)((\nu+1)c-(\nu+\lambda))
$$
Because of the condition $c>\max\{1,\Frac{\nu+\lambda}{\nu+1}\}$ 
the first and last coefficients of the polynomial are positive. Then, if we prove that the second coefficient is negative
we conclude that the two roots of $\tilde{R}(\sigma)$ are positive. This is shown using that 
$c<c_\nu^{(I)}(\lambda)$, which gives
$$
\begin{array}{ll}
(3\nu+\lambda+1)c-3\nu-3\lambda&<(3\nu+\lambda+1)c_\nu^{(I)}(\lambda)-3\nu-3\lambda=\\
&\\
&=\Frac{2(\lambda+\nu)(\sqrt{2\lambda}-2)(\sqrt{2\lambda}-1)}{\nu-\lambda+2\sqrt{2\lambda}-1}<0
\end{array}
$$
where the last inequality holds for $\lambda\in (1/2,2)$.
\end{proof}

\subsubsection{Close to best lower bound for {\boldmath $K_{\nu+1}(x)/K_{\nu}(x)$}}

We end the description of the close to best bounds with the result analogous to Theorem \ref{IU} 
but for the second kind Bessel function. 
The proof, which is omitted, 
is very similar to that of 
Theorem \ref{IU}, but with the difference that in this proof we start by comparing the ratio of Bessel 
functions with the bound at $+\infty$ and not at $x=0$ (as happened in the proof of Theorem \ref{bnk}).

\begin{theorem} 
\label{blk}
The following holds for  $\lambda\in [1/2,2]$, $\nu\ge \lambda$ and $x>0$
$$
\Frac{K_{\nu+1}(x)}{K_{\nu}(x)}>L_{\nu}^{(K)}(\lambda,x)=B(\nu+\lambda,\nu-\lambda, \sqrt{c_{\nu}^{(K)}(\lambda)}, x),
$$
where
$$
c_\nu^{(K)}(\lambda)=\Frac{\nu-\lambda}{\nu+\lambda-2\sqrt{2\lambda}+1}.
$$

\end{theorem}

\begin{remark}
The mirror symmetry mentioned in Remark \ref{simeta} is again noticeable when we compare Theorems \ref{IU} and \ref{blk}.
\end{remark}

\begin{remark}
In the previous theorem the condition $\nu\ge \lambda$ is necessary so that $c_\nu^{(K)}(\lambda)\ge 0$ (otherwise 
that bound can not hold for all $x>0$). The validity of the bound can be extended to positive $\lambda$, however the 
bounds when $\lambda\in [1/2,2]$ are sharper than for values outside this interval.
\end{remark}

\section{Best possible bounds}
\label{bestbo}

From the previous close to best bounds, we can deduce the existence of best possible bounds which 
interpolate the function ratios at $x=0$ or $x=+\infty$ and at an intermediate value $x_*>0$, with interpolating
conditions up to the first derivative. In \cite{Hornik:2013:ABF} the corresponding lower bound was shown for 
the ratio $I_{\nu-1}(x)/I_{\nu}(x)$. Here we prove that the same type of result can be given for the upper bound
for this ratio and for the lower and upper bounds for the ratio $K_{\nu+1}(x)/K_{\nu}(x)$. 

The idea for proving these best bounds is to pull down (for upper bounds) or up (for lower bounds) the 
close to best bound
by changing one of the coefficients until the bound stops being a bound.
Just before this happens, the graph of the bound will be tangent to the graph of the 
ratio of Bessel functions at some positive value of
$x=x_*$, and the resulting osculatory bound will be the best possible around such value of $x_*$.

The following lemmas are immediate to prove and they will be used later. They give osculatory 
functions which are defined for all $x>0$. What we will later prove is that such functions
are in fact bounds (the best bounds). 

We start with the functions which are sharp at $x=+\infty$.

\begin{lemma}
\label{unic}
Let $\phi^-_{\nu} (x)=xI_{\nu-1}(x)/I_\nu (x)$ and $\phi^+_{\nu} (x)=xK_{\nu+1}(x)/K_\nu (x)$, 
$\nu\ge 1/2$, and define
\begin{equation}
\label{tan0}
h^{\pm}_{\nu}(a,b,x)=a+\sqrt{b+x^2}.
\end{equation}
Given $x_*>0$ there exist unique values
$a^{\pm}_{*}$ and $b^{\pm}_{*}$ such that
\begin{equation}
\label{condit0}
\phi^{\pm}_{\nu} (x_*)
=h^{\pm}_{\nu}(a^\pm_*,b^\pm_*,x_*),\,\phi^{\pm \prime}_\nu  (x_*)=
h^{\pm \prime}_{\nu}(a^\pm_*,b^\pm_*.x_*).
\end{equation}
As functions of the tangency point $x_*\in \mathbb{R}^+$, $a^\pm_*$ and $b^-_*$ 
can be written as follows
\begin{equation}
\label{ab}
\begin{array}{l}
a^\pm_*=\phi^\pm_\nu(x_*)-\Frac{x_*}{\phi^{\pm \prime}_\nu(x_*)},\\
\\
b^\pm_*=x_*^2\left(\phi^{\pm \prime}_\nu(x_*)^{-2}-1\right).
\end{array}
\end{equation}

Additionally we have $b^{\pm}_*>0$ and therefore $h^{\pm}_{\nu}(a^\pm_*,b^\pm_*,x)$ 
is real for all $x\ge 0$.
\end{lemma}

\begin{proof}
From the conditions (\ref{condit0}) we easily get (\ref{ab}), and considering Theorem \ref{gapk} we
see that $b^\pm_*>0$.
\end{proof}

Next, we consider functions $B(\alpha,\beta,\gamma,x)$ which are sharp at $x=0$, that is, with
$\alpha+\beta=2\nu$.

\begin{lemma}
\label{uniq}
Let $\phi^-_{\nu} (x)=xI_{\nu-1}(x)/I_\nu (x)$ and $\phi^+_{\nu} (x)=xK_{\nu+1}(x)/K_\nu (x)$ and define
\begin{equation}
\label{tan1}
h^{\pm}_{\nu}(\lambda,c,x)=\nu\pm\lambda+\sqrt{(\nu\mp \lambda)^2+cx^2}.
\end{equation}
Given $x_*>0$ there
exist unique values $\lambda^{\pm}_{*}$ and $c^{\pm}_{*}$ such that
\begin{equation}
\label{condit}
\phi^{\pm}_{\nu} (x_*)
=h^{\pm}_{\nu}(\lambda^\pm_*,c^\pm_*,x_*),\,\phi^{\pm \prime}_\nu  (x_*)=
h^{\pm \prime}_{\nu}(\lambda^\pm_*,c^\pm_*.x_*).
\end{equation}
As functions of the tangency point $x_*\in \mathbb{R}^+$, $\lambda^-_*$ and $c^-_*$ 
can be written in terms of $\phi_*=\phi^-_*(x_*)$ as follows
$$
\begin{array}{l}
\lambda^-_*=
-\Frac{\phi_*^3+(1-3\nu)\phi_*^2+(2\nu (\nu-1)-x_*^2)\phi_*+\nu x_*^2}{\phi_*^2+2(1-\nu)\phi_*-x_*^2-4\nu},\\
\\
c^-_*=-\Frac{(\phi_*^2-2\nu \phi_*-x_*^2)(\phi_*-2\nu)^2}{x_*^2(\phi_*^2+2(1-\nu)\phi_*-x_*^2-4\nu)}.
\end{array}
$$
The expressions for $\lambda^+_*$ and $c^+_*$ are the same, 
replacing $\nu$ by $-\nu$ in the right-hand side and
taking $\phi_*=-\phi^+_\nu (x_*)$.

Additionally we have $c^{\pm}_*>0$ and therefore $h^{\pm}_{\nu}(\lambda^\pm_*,c^\pm_*,x)$ 
is real for all $x\ge 0$.
\end{lemma}
\begin{proof}
We just have to use the two conditions (\ref{condit}) for determining $\lambda_*$ and $c_*$. 

Let us consider the case of $\phi^-_\nu(x)$; for $\phi^+_\nu(x)$ the derivation is analogous.
Denoting $S=\sqrt{(\nu+\lambda_*^-)^2+c^-_* x_*^2}$, the conditions (\ref{condit}) give
\begin{equation}
\label{primo}
\nu-\lambda_*^-+S=\phi_*,\,\Frac{c^-_*x_*}{S}=\phi^{-\prime}_{\nu}(x_*)
\end{equation}
Eliminating $S$ in (\ref{primo})
\begin{equation}
\label{sego}
x_*^2 c^-_*=x_*\phi^{-\prime}_{\nu}(x_*)(\phi_*^--\nu+\lambda_*^-)
\end{equation}
and squaring the first equality in (\ref{primo})
$$
x_*^2 c^-_* =\phi_*^{-\,2}-2\nu\phi_*^{-}+2\lambda_*^- (\phi_*^{-}-2\nu).
$$
From the last two equations we get $\lambda_*^-$ and $c_*^-$. 
Additionally, using 
$x_* \phi^{- \prime}_*  (x_*)=x_*^2+2\nu\phi_*^--\phi_*^{-\,2}$, we get the final expression. Similarly
for $\phi^+_\nu(x)$.

Now, the fact that $c^\pm_* >0$ is a consequence of the bounds of types $(1,1)$ and $(2,1)$ of tables 
\ref{tablei} and \ref{tablek}.
\end{proof}

\subsection{Best lower bound for {\boldmath $I_{\nu-1}(x)/I_{\nu}(x)$}}

The next result is an adaptation of \cite[Theorem 10]{Hornik:2013:ABF}, and therefore does not need
 a proof. Later (Theorem \ref{bestboku}) we prove a similar result for second kind Bessel functions.
 These two theorems will characterize 
 the best bounds around some $x>0$ which are also sharp at $x=+\infty$.
 In theorems \ref{bestboi} and \ref{bestbokuku} we will close the analysis with the best bounds around
 any given $x>0$ which are also sharp at $x=0$. These four theorems will give the best possible upper
 and lower bounds of the type $B(\alpha,\beta,\gamma,x)$ for the ratios of both the first and second
 kind Bessel functions.

\begin{theorem} 
\label{bestboil}
Let $\lambda\in (0,1/2)$ and $\nu\ge \frac12-\lambda$, then there exists a value ${\cal B}_{\nu}^{(I)}(\lambda)>0$ such that
$$
{\cal L}_{\nu}^{(I)}(\lambda,x)=\Frac{1}{x}\left(\nu-1/2-\lambda +\sqrt{{\cal B}_{\nu}^{(I)}(\lambda)+ x^2}\right)
$$
satisfies
$$
h_{\nu}(x)=\Frac{I_{\nu-1}(x)}{I_{\nu}(x)}\ge {\cal L}_{\nu}^{(I)}(\lambda,x),x>0
$$
where, for fixed $\nu$ and $\lambda$, 
the equality holds at one and only one value of the variable $x_*=x_\nu^{(I)}(\lambda)>0$, where 
$h_{\nu}'(x_*)={\cal L}_\nu^{(I) \prime}(\lambda,x_*)$.

As a function of $\lambda$, ${\cal B}_\nu^{(I)} (\lambda)$ is increasing while $x_\nu^{(I)} (\lambda)$ is decreasing and the following limits
hold:
$$
\displaystyle\lim_{\lambda\rightarrow 0}{\cal B}_\nu^{(I)}(\lambda)=\nu^2-\frac14,
\,\displaystyle\lim_{\lambda\rightarrow 1/2}{\cal B}_\nu^{(I)} (\lambda)=(\nu+1)^2,
$$
$$
\displaystyle\lim_{\lambda\rightarrow 0}x_\nu^{(I)}(\lambda)=+\infty,\,
\displaystyle\lim_{\lambda\rightarrow 1/2}x_\nu^{(I)}(\lambda)=0
$$
\end{theorem}

\subsection{Best upper bound for {\boldmath $K_{\nu+1}(x)/K_{\nu}(x)$}}

\begin{theorem} 
\label{bestboku}
Let $\nu>1/2$ and $0<\lambda<a$, $a=\min\{\frac12,\nu-\frac12\}$, then there exists a value ${\cal B}_{\nu}^{(K)}(\lambda)>0$ such that
$$
{\cal U}_{\nu}^{(K)}(\lambda,x)=\Frac{1}{x}\left(\nu+1/2+\lambda +\sqrt{{\cal B}_{\nu}^{(K)}(\lambda)+ x^2}\right)
$$
satisfies
$$
h_{\nu}(x)=\Frac{K_{\nu+1}(x)}{K_{\nu}(x)}\le {\cal U}_{\nu}^{(K)}(\lambda,x),x>0
$$
where, for fixed $\nu$ and $\lambda$, 
the equality holds at one and only one value of the variable $x_*=x_\nu^{(K)}(\lambda)>0$, where 
$h_{\nu}'(x_*)={\cal U}_\nu^{(K) \prime}(\lambda,x_*)$.

As a function of $\lambda$, both ${\cal B}_\nu^{(K)} (\lambda)$ and $x_\nu^{(K)} (\lambda)$ are decreasing and the following limits
hold for $\nu>1/2$:
$$
\displaystyle\lim_{\lambda\rightarrow 0}{\cal B}_\nu^{(K)}(\lambda)=\nu^2-\frac14,
\,\displaystyle\lim_{\lambda\rightarrow a}{\cal B}_\nu^{(K)} (\lambda)=\Theta(\nu-1)(\nu-1)^2,
$$
$$
\displaystyle\lim_{\lambda\rightarrow 0}x_\nu^{(K)}(\lambda)=+\infty,\,
\displaystyle\lim_{\lambda\rightarrow a}x_\nu^{(K)}(\lambda)=0
$$
with $\Theta (x)=1$ if $x\ge 0$ and $\Theta (x)=0$ if $x< 0$
\end{theorem}

\begin{proof}

For brevity, in the proof we denote $x_* =x_{\nu}^{(K)}(\lambda)$, $b_*={\cal B}_\nu^{(K)}(\lambda)$. 

In addition, we denote $\delta_{\nu}(\lambda,b,x)=u_\nu(\lambda,b,x)-\phi_{\nu}(x)$ where
$\phi_{\nu}(x)=x h_{\nu}(x)$ and $u_{\nu}(\lambda,b,x)=\nu+\lambda+1/2+\sqrt{b+ x^2}$.
 The conditions $h_{\nu}(x_*)=
{\cal U}_{\nu}^{(K)}(\lambda,x_*)$ and $h_{\nu}'(x_*)={\cal U}_{\nu}^{(K)\prime}(\lambda,x_*)$ 
are equivalent to $\delta_{\nu}(\lambda,b_*,x_*)=0$ 
and $\delta_{\nu}'(\lambda,b_*,x_*)=0$ 

We start by noticing that, according to Theorem \ref{bnk}, for $\nu>1/2$ and for each $\lambda \in (0,1/2)$ there exist a value 
$b(\lambda)=\beta_{\nu}^{(K)}(\lambda)^2$ such that
 $\delta_{\nu}(\lambda,b(\lambda),x)>0$ for all $x> 0$. On the other hand,  $u_\nu(\lambda,b,x)$ 
 decreases as $b$ decreases, and, as we see next, under the conditions of the theorem there exists a
 minimal value $b_*$ such that  $\delta_{\nu}(\lambda,b_*,x)\ge 0$ for all $x>0$, where the equality
 holds for only one positive value $x=x_*>0$.
 
 First we notice that
 $$
 \delta_{\nu}(\lambda,b,x)=\lambda +\Frac{b-(\nu^2-1/4)}{2x}+{\cal O}(x^{-2}),
 $$
 and therefore the bound remains an upper bound for large enough $x$ (and it is sharp as $x\rightarrow +\infty$) when $\lambda>0$.
 As $\lambda\rightarrow 0^+$ $\delta_{\nu}(\lambda,b,x)$ can not remain positive for all $x>0$ unless 
 $b>\nu^2-1/4$, otherwise $\delta_{\nu}(\lambda,b,x)$ could
 become negative for large $x$ (but not too large). Therefore we conclude that $b_*>\nu^2-1/4$ for small enough $\lambda>0$.

 On the other hand, as $x\rightarrow 0^+$ we have:
 $$
 \delta_{\nu}(\lambda,b,x)=\sqrt{b}+\lambda+1/2-\nu +{\cal O}(x^2).
 $$
 We observe that if $\sqrt{b}<\nu-\lambda-1/2$ (and $\lambda<\nu-1/2$ so that $\sqrt{b}$ can
 be positive\footnote{Observe that the condition $\lambda <a=\min\{1/2,\nu-1/2\}$ is one of
 the hypotheses of the theorem}) 
 $\delta_{\nu}(\lambda,b,0)<0$ and therefore we no longer have an upper bound. From this, we see that $\sqrt{b_*}\ge \nu-\lambda-1/2$ and
 in fact the inequality must be strict, as we next check. Setting $\sqrt{b}=\nu-\lambda-1/2$ we have
 $$
 \delta_{\nu}(\lambda,\nu-\lambda-1/2,x)=\Frac{\lambda-1/2}{2(\nu-\lambda-1/2)(\nu-1)}x^2+{\cal O}(x^4),\,\nu>1
 $$
 and
 $$
 \delta_{\nu}(\lambda,\nu-\lambda-1/2,x)=-2\nu\Frac{\Gamma (1-\nu)}{\Gamma (1+\nu) }\left(\Frac{x}{2}\right)^{2\nu}
 +{\cal O}(x^2),\,\nu\in (0,1).
 $$
 In both cases, because $\lambda \in (0,1/2)$, we have that $\delta_{\nu}(\lambda,\nu-\lambda-1/2,0)=0$ but 
 $\delta_{\nu}(\lambda,\nu-\lambda-1/2,x)<0$ for small enough $x$. Therefore $u_\nu(\lambda,b,x)$ is no longer an
 upper bound for $\phi_{\nu}(x)$ at least in some positive interval; by continuity, the same holds for $\nu=1$. 
 We conclude that $\sqrt{b_*}>\nu-\lambda-1/2$ and therefore
 \begin{equation}
 \label{bob*}
\nu-\lambda-1/2<\sqrt{b_* (\lambda)}<\beta_{\nu}^{(K)}(\lambda).
 \end{equation}

 We conclude that if $\nu>1/2$  and $0<\lambda<a$, $a=\min\{\frac12,\nu-\frac12\}$ this value 
 $b_*>0$ does exist. Furthermore, for if $b=b_*$ there must exist at least one value of $x=x_*>0$ such that 
 $\delta_{\nu}(\lambda,b_*,x_*)=0$ and $\delta_{\nu}'(\lambda,b_*,x_*)=0$; indeed, because 
 $b_*$ is the minimal value of $b$ for which $\delta_\nu(\lambda,b,x)\ge 0$
 for all $x\ge 0$, there must 
 exist at least one value
 $x_*\ge 0$ such that $\delta_\nu(\lambda,b_*,x_*)=0$ and  $\delta_\nu(\lambda,b_*-\epsilon,x_*) 
\delta_\nu(\lambda,b_*+\epsilon,x_*)<0$ for sufficiently small $\epsilon$. From or previous discussion it is clear
that $x_*\neq 0$ and because 
$\delta_\nu(\lambda,b_*,x_*)= 0$ and $\delta_\nu(\lambda,b_*,x)\ge 0$
for all $x\ge 0$
necessarily $\delta^{\prime}_\nu(\lambda,b_*,x_*)=0$.

The previous discussion proved that there exists a point of tangency $x_* (\lambda)$ for each $\lambda$.
This point, on the other hand, must be a solution of the equation (\ref{qs}) with 
$\beta=\sqrt{b_* (\lambda)}$, and because we have (\ref{bob*}) we are in the conditions of Lemma \ref{smayb},
 which means that both $s$-solutions are greater that  $\beta=\sqrt{b_* (\lambda)}$ and, therefore, 
 because $s=\sqrt{\beta^2+x^2}$ there are two positive real solutions $x_1 (\lambda)$ and $x_2 (\lambda)$
 (and they are different because the discriminant of $Q(s)=0$ is positive). One of this two solutions gives
 the tangency point and the other one plays no role, as we later prove.
 
 Before this, we prove that $b_*$ is decreasing as a function of $\lambda$. This is a consequence of the fact that $u_\nu(\lambda,b,x)$
 increases both as a function of $\lambda$ and $b$. We assume the contrary and we arrive at a contradiction: we take
  $\lambda_1<\lambda_2$ and we suppose that $b_* (\lambda_1)\le b_* (\lambda_2)$, which implies 
  that $u_\nu(\lambda_1,b_* (\lambda_1),x)<u_\nu(\lambda_2,b_* (\lambda_2),x)$; 
 but because $u_\nu(\lambda_1,b_* (\lambda_1),x)$ is an upper bound for $\phi_{\nu}(x)$ then $u_\nu(\lambda_2,b_* (\lambda_2),x)$ can 
 not have a tangency point with $\phi_{\nu}(x)$, in contradiction with the definition of $b_* (\lambda_2)$.

With respect to the limits as $\lambda\rightarrow a$, we first consider the case $\nu\ge 1$, which implies 
$a=1/2$. Taking into account (\ref{bob*}) we have
$\lim_{\lambda\rightarrow 1/2}\sqrt{b_* (\lambda)}=\nu-1$. On the other hand, using this limit in
(\ref{eqs}) we deduce that both $s$-roots are such that $\lim_{\lambda\rightarrow 1/2}s(\lambda)=\nu-1$, and
therefore $\lim_{\lambda\rightarrow 1/2}x_* (\lambda)=0$. 

For $\nu\in (1/2,1)$, because we have the condition $\lambda<\nu-1/2$, we must consider the limits
$\lambda\rightarrow \nu-1/2$. In this case, taking into account Lemma \ref{limice}, both roots
$s(\lambda)$ of (\ref{eqs}) tend to zero as $\lambda\rightarrow a$, and because the variable $s$ is
defined as (\ref{ese}), this means that $s(\lambda)=\sqrt{b_*(\lambda)+x_*(\lambda)^2}$, 
tends
to zero, with $b_*(\lambda)>0$, and therefore we have that $b_*(\lambda)$ and $x_*(\lambda)$
tend to zero.

Now, we investigate the limits $\lambda\rightarrow 0$. Because by construction we 
know that $b_*(\lambda)<{\cal B}_{\nu}^{(K)}(\lambda)$ and,
as discussed earlier in the proof, $b_*(\lambda)>\nu^2-1/4$ as $\lambda\rightarrow 0$, for small
$\lambda$ we have
$$
\nu^2-1/4<b_*(\lambda)<{\cal B}_{\nu}^{(K)}(\lambda)
$$
and
$$
\displaystyle\lim_{\lambda\rightarrow 0}b_*(\lambda)=\nu^2-1/4.
$$
In addition, using (\ref{eqs}) we see that both roots of the equation tend to infinity in this limit.

Finally, we prove that for each $\lambda$, the point of tangency $x_*\equiv x_*(\lambda)$ is unique and decreasing as a function of $\lambda$. As we mentioned before in this same proof, the point of
 tangency must be one of the solutions of $Q(s)=0$, which has two distinct positive $x$-solutions. We have
 already proved that

\begin{equation}
\label{limxk*}
 \displaystyle\lim_{\lambda\rightarrow 0^+} x_* (\lambda)=+\infty,\,\displaystyle\lim_{\lambda\rightarrow a} x_* (\lambda)=0.
 \end{equation}
 
 Let us denote the two solutions as $x_*^{(1)}(\lambda)$ and $x_*^{(2)}(\lambda)$, with 
$x_*^{(2)}(\lambda)>x_*^{(1)}(\lambda)$. Only one of these solutions gives the point of tangency, which is therefore unique. 
 
Indeed, given a value of the tangency point $x_*$, the corresponding value of 
$\lambda=\lambda(x_*)$ is unique for any $x_*$ (Lemma \ref{unic}); however, because of (\ref{limxk*}) and the 
continuity of $x_*^{(i)}(\lambda)$, $i=1,2$, for each $x_*>0$ there are at least two values of $\lambda$, say $\lambda_1$ and
$\lambda_2$, $\lambda_1\neq \lambda_2$, such that 
$x_*^{(1)}(\lambda_1)=x_*$ and $x_*^{(2)}(\lambda_2)=x_*$; either $\lambda_* =\lambda_1$ or $\lambda_*=\lambda_2$, but not both, 
and therefore one of the solutions $x_*^{(i)}(\lambda)$, $i=1,2$ plays no role. On the other hand, it is not possible that 
the tangency point is given by $x_*^{(1)}(\lambda)$ or $x_*^{(2)}(\lambda)$ depending on the value of $\lambda$, because 
$x_*^{(2)}(\lambda)>x_*^{(1)}(\lambda)$, and the tangency point $x_*$ must be continuous as a function of $\lambda$. 
We have checked numerically that the tangency point is given by the smaller root: $x_* (\lambda)=x_*^{(º)} (\lambda)$ (this fact
 is not necessary for proving this theorem and we have not pursued its proof).

Finally, because
of the limits (\ref{limxk*}) and the continuity of $x_* (\lambda)$,   it can be proved that $x_* (\lambda)$ is decreasing by checking that the correspondence $x_* \rightarrow \lambda (x_*)$
is injective, which is, because given $x_*$, 
the two conditions of tangency  $\delta_{\nu}(\lambda,b_*,x_*)=0$ and $\delta_{\nu} '(\lambda,b_*,x_*)=0$ univocally determine $\lambda=\lambda (x_*)$ and $d_*(\lambda)$.
  
\end{proof} 

\subsection{Best upper bound for {\boldmath $I_{\nu-1}(x)/I_{\nu}(x)$}}

\begin{theorem}
\label{bestboi}
Let $\nu\ge 0$. For each $\lambda\in (1/2,2)$, there exists a value ${\cal C}_{\nu}^{(I)}(\lambda)>0$ such that
$$
{\cal U}_{\nu}^{(I)}(\lambda,x)=\Frac{1}{x}\left(\nu-\lambda +\sqrt{(\nu+\lambda)^2+ 
{\cal C}_{\nu}^{(I)}(\lambda)\,x^2}\right)
$$
satisfies
$$
h_{\nu}(x)=\Frac{I_{\nu-1}(x)}{I_{\nu}(x)}\le {\cal U}_{\nu}^{(I)}(\lambda,x),x>0
$$
where, for fixed $\nu$ and $\lambda$, 
the equality holds at one and only one value of the variable $x_*=x_\nu^{(I)}(\lambda)>0$, where 
$h_{\nu}'(x_*)={\cal U}_\nu^{(I) \prime}(\lambda,x_*)$.

As a function of $\lambda$, ${\cal C}_\nu^{(I)} (\lambda)$ is increasing while $x_\nu^{(I)} (\lambda)$ is decreasing and the following limits
hold:
$$
\displaystyle\lim_{\lambda\rightarrow 1/2}{\cal C}_\nu^{(I)}(\lambda)=1,\,\displaystyle\lim_{\lambda\rightarrow 2}{\cal C}_\nu^{(I)} (\lambda)=\Frac{\nu+2}{\nu+1},
$$
$$
\displaystyle\lim_{\lambda\rightarrow 1/2}x_\nu^{(I)}(\lambda)=+\infty,\,
\displaystyle\lim_{\lambda\rightarrow 2}x_\nu^{(I)}(\lambda)=0
$$
\end{theorem}

\begin{proof}
For brevity, in the proof we denote $x_* =x_{\nu}^{(I)}(\lambda)$, $c_*={\cal C}_\nu^{(I)}(\lambda)$. 

In addition, we denote $\delta_{\nu}(\lambda,c,x)=b_\nu(\lambda,c,x)-\phi_{\nu}(x)$ where
$\phi_{\nu}(x)=x h_{\nu}(x)$ and $b_{\nu}(\lambda,c,x)=\nu-\lambda+\sqrt{(\nu+\lambda)^2+c x^2}$.
 The conditions $h_{\nu}(x_*)=
{\cal U}_{\nu}^{(I)}(\lambda,x_*)$ and $h_{\nu}'(x_*)={\cal U}_{\nu}^{(I)\prime}(\lambda,x_*)$ 
are equivalent to $\delta_{\nu}(\lambda,c_*,x_*)=0$ 
and $\delta_{\nu}'(\lambda,c_*,x_*)=0$ 

We start by noticing that, according to Theorem \ref{IU}, for each $\lambda \in (1/2,2)$ there exist a value 
$c_\nu^{(I)}(\lambda)$ such that
 $\delta_{\nu}(\lambda,c_\nu^{(I)}(\lambda),x)>0$ for all $x> 0$. On the other hand, because 
$b_\nu(\lambda,c,x)$ decreases as $d$ decreases 
 and 
 $\delta_{\nu}(\lambda,c_\nu^{(I)}(\lambda),x)>0$ 
 while $\delta_{\nu}(\lambda,0,x)<0$  (because 
 $I_{\nu-1}(x)/I_{\nu}(x)>2\nu/x$) there must 
 be a value $c_*>0$ such that $b_\nu(\lambda,c,x)>0$ for $c>c_*$ for all $x$ but that this does not
 hold for $c<c_*$. 

We have that $c_*$ is the minimal value of $c$ for which $\delta_\nu(\lambda,c,x)\ge 0$
for all $x\ge 0$,
and then
there must 
 exist at least one value
 $x_*\ge 0$ such that $\delta_\nu(\lambda,c_*,x_*)=0$ and  $\delta_\nu(\lambda,c_*-\epsilon,x_*) 
\delta_\nu(\lambda,c_*+\epsilon,x_*)<0$ for sufficiently small $\epsilon$. We have 
$x_* \neq 0$ because  $\delta_\nu(\lambda,c,0)=0$ for all $c$. Then, $x_*>0$ and because 
$\delta_\nu(\lambda,c_*,x_*)= 0$ and $\delta_\nu(\lambda,c_*,x)\ge 0$
for all $x\ge 0$
necessarily $\delta^{\prime}_\nu(\lambda,c_*,x_*)=0$. In other words:
$$
h_{\nu}(x_*)=U_\nu^{(I)}(\lambda,x_*),\,h_{\nu}'(x_*)=U_\nu^{(I) \prime}(\lambda,x_*).
$$
 
 By construction, an upper bound for $c_*$ is $c_{\nu}^{(I)}(\lambda)$ and we can find lower bounds by comparing
 (\ref{expaba}) (with the selection of $\alpha$, $\beta$ and $\gamma$ in this theorem) with (\ref{expin}) 
 and (\ref{serin}).

 The first term in both expansions (\ref{expaba}) and (\ref{serin}) coincides, while the second term is greater for 
 $b_\nu(\lambda,c,x)$ only if 
 $c>(\nu+\lambda)/(\nu+1)$, while for $c=(\nu+\lambda)/(\nu+1)$ we have 
 $$
 \delta_\nu(\lambda,c,x)=\Frac{\lambda -2}{(\nu+\lambda)(\nu+1)^2(\nu+2)}x^4+{\cal O}(x^6).
 $$
 Therefore, because $\lambda\in (1,2)$, $\delta_\nu(\lambda,c_*,x)<0$ close to $x=0$ if $c\le (\nu+\lambda)/(\nu+1)$ 
 and then 
 $b_\nu(\lambda,c,x)$ is no longer an upper bound for such values; hence $c_*>(\nu+\lambda)/(\nu+1)$. 
 On the other hand comparing the expansions
 as $x\rightarrow +\infty$ using (\ref{expin}), we conclude that $\delta_\nu(\lambda,c,x)<0$ if $c<1$. Therefore $c_*\ge 1$ so that
 $b_\nu(\lambda,c_*,x)$ can be an upper bound. The value $c=1$ is also excluded because in this case
 $$
 \delta_\nu(\lambda,c,x)=\Frac{\frac12-\lambda}{x}+{\cal O}(x^{-2}).
 $$
 and then $\delta_\nu(\lambda,c,x)<0$ for large $x$.

 With this we conclude that
 \begin{equation}
 \label{bod}
 \max\left\{1,\Frac{\nu+\lambda}{\nu+1}\right\}<c_*<c_{\nu}^{(I)}(\lambda).
 \end{equation}
 In addition, because $c_\nu^{(I)}(1/2)=1$ and $c_{\nu}^{(I)}(2)=(\nu+2)/(\nu+1)$ we
 have, using the previous bounds, that $\lim_{\lambda \rightarrow 1/2}c_*=1$ and 
 $\lim_{\lambda \rightarrow 2}c_* =(\nu+2)/(\nu+1)$.

Now we prove that $c_*\equiv c_*(\lambda)$ is increasing as a
function of $\lambda$. We notice that $b_\nu(\lambda,c,x)$ is decreasing as a function of $\lambda$ and increasing
as a function of $c$. Then if
$\lambda_1<\lambda_2$ and $c_*(\lambda_1)\ge c_*(\lambda_2)$ we would have $b_\nu(\lambda_1,c_*(\lambda_1),x)>
b_\nu(\lambda_2,c_*(\lambda_2),x)\ge \phi_{\nu}(x)$ for al $x\ge 0$, but in the case there can 
not be a value $x_*$ of tangency for $b_\nu(\lambda_1,c_*(\lambda_1),x)$, in contradiction with the definition of the minimal value
$c_*(\lambda_1)$. Therefore, if $\lambda_1<\lambda_2$ then $d(\lambda_1)<d(\lambda_2)$.

Finally, we prove that for each $\lambda$, the point of tangency $x_*\equiv x_*(\lambda)$ is unique and decreasing as a function of $\lambda$.
First, denoting 
\begin{equation}
\label{relo}
s_*(\lambda)=\sqrt{(\nu+\lambda)^2+c_*(\lambda)x_*(\lambda)^2}
\end{equation}
from the analysis of Lemma \ref{lemaprep}, we now
that $s_*\equiv s_*(\lambda)$ must be solution of the equation $R(s_*)=0$ (see Eq. (\ref{rs})) with
$c=c_*(\lambda)$. 
We notice that,
for $\lambda \in (1/2,2)$ there are two different real roots $s_*$ of (\ref{rs})  
 because the discriminant is positive (see Lemma \ref{signodis}) and both roots are greater than $\nu+\lambda$ (Lemma 
 \ref{signosol}). Now, using (\ref{rs}) we see that both solutions satisfy
 \begin{equation}
\label{lims*}
 \displaystyle\lim_{\lambda\rightarrow 1/2} s_* (\lambda)=+\infty,\,\displaystyle\lim_{\lambda\rightarrow 2} s_* (\lambda)=
  \nu+\lambda.
 \end{equation}
Now, considering (\ref{relo}) we see that there are two distinct positive real solutions of (\ref{rs}) in terms of $x_*$ and that
both satisfy
\begin{equation}
\label{limx*}
 \displaystyle\lim_{\lambda\rightarrow 1/2} x_* (\lambda)=+\infty,\,\displaystyle\lim_{\lambda\rightarrow 2} x_* (\lambda)=0.
 \end{equation}
 
 Because, as discussed, the discriminant is positive one of the solutions will be larger than the other one for all 
 $\lambda\in (1/2,2)$. Let us denote the solutions as $x_*^{(1)}(\lambda)$ and $x_*^{(2)}(\lambda)$, with 
 $x_*^{(2)}(\lambda)>x_*^{(1)}(\lambda)$. Only one of these solutions gives the point of tangency, which is therefore unique. This is proved similarly as was done in Theorem \ref{bestboi}, now considering Lemma \ref{uniq}. We have checked numerically that the tangency point is given by the larger root: $x_* (\lambda)=x_*^{(2)} (\lambda)$ (this fact
 is not necessary for proving this theorem). Also,  similarly as in in Theorem \ref{bestboi} it
 follows that the tangency point $x_* (\lambda)$ is decreasing as a function of $\lambda$

\end{proof}

\subsection{Best lower bound for {\boldmath $K_{\nu+1}(x)/K_{\nu}(x)$}}

\begin{theorem}
\label{bestbokuku}
Let $\lambda\in (1/2,2)$ and $\nu\ge \lambda$. There exists a value 
${\cal C}_{\nu}^{(K)}(\lambda)>0$ such that
$$
{\cal L}_{\nu}^{(K)}(\lambda,x)=\Frac{1}{x}\left(\nu+\lambda +\sqrt{(\nu-\lambda)^2+ 
{\cal C}_{\nu}^{(K)}(\lambda)\,x^2}\right)
$$
satisfies
$$
h_{\nu}(x)=\Frac{K_{\nu+1}(x)}{K_{\nu}(x)}\ge {\cal L}_{\nu}^{(K)}(\lambda,x),x>0
$$
where, for fixed $\nu$ and $\lambda$, 
the equality holds at one and only one value of the variable $x_*=x_\nu^{(K)}(\lambda)>0$, where 
$h_{\nu}'(x_*)={\cal L}_\nu^{(K) \prime}(\lambda,x_*)$.

As a function of $\lambda$, both ${\cal C}_\nu^{(K)} (\lambda)$ and $x_\nu^{(K)} (\lambda)$ are decreasing and the following limits
hold:
$$
\displaystyle\lim_{\lambda\rightarrow 1/2}{\cal C}_\nu^{(K)}(\lambda)=1,\,\displaystyle\lim_{\lambda\rightarrow 2}{\cal C}_\nu^{(K)} (\lambda)=\Frac{\nu-2}{\nu-1},
$$
$$
\displaystyle\lim_{\lambda\rightarrow 1/2}x_\nu^{(K)}(\lambda)=+\infty,\,
\displaystyle\lim_{\lambda\rightarrow 2}x_\nu^{(K)}(\lambda)=0
$$
\end{theorem}

\subsection{Summary of the best bounds}

Finally, we summarize the best bounds considered in this section, skipping some details on the
range of validity, and we show the relation with the best bounds at $x=0$ and $x=+\infty$ (the
bounds with accuracies $(1,2)$, $(3,0)$, $(2,1)$ and $(0,3)$) \footnote{
Animated images for these best bounds (as a function of $x_*$) and for the close to best bounds
(as a function of $\lambda$) are available at http://personales.unican.es/segurajj/bounds.html, 
showing how these bounds evolve from the best bound at $x=0$ to the best bound at 
$x=+\infty$ with intermediate stages which constitute best (or close to best) 
bounds around the tangency point.
}. 
 
\begin{theorem}
Let $h_\nu (x)=I_{\nu-1}(x)/I_{\nu}(x)$ or 
$h_\nu (x)=K_{\nu+1}(x)/K_{\nu}(x)$. 
Let 
$x_*>0$ and 
$B(\alpha,\beta,\gamma,x)=
(\alpha+\sqrt{\beta^2+\gamma^2 x^2})/x$, with $\alpha$, $\beta$ and $\gamma$ 
determined by the following three 
conditions:
\begin{enumerate}
\item{Interpolatory conditions at $x_*$:} $h_\nu (x_*)=B(\alpha,\beta,\gamma,x_*)$, 
$h'_\nu (x_*)=B'(\alpha,\beta,\gamma,x_*)$
\item{Sharpness condition:}
$\displaystyle\lim_{x\rightarrow x_s}h_\nu (x)/B(\alpha,\beta,\gamma,x)=1$, where either $x_s=0^+$ or 
$x_s=+\infty$. 
\end{enumerate}
Denoting $B_{\nu}^{(j)}(x_s,x_*,x)=B(\alpha,\beta,\gamma,x)$, where
we assign the label $j=I$ for the case $h_\nu (x)=I_{\nu-1}(x)/I_{\nu}(x)$ and $j=K$ for 
$h_\nu (x)=K_{\nu+1}(x)/K_{\nu}(x)$, the following holds for $x>0$:
\begin{enumerate}
\item{}$B^{(I)}_\nu (+\infty,x_*,x)\le I_{\nu-1}(x)/I_{\nu}(x) \le 
B^{(I)}_\nu (0,x_*,x)$. 
\item{} $B^{(K)}_\nu (0,x_*,x)\le K_{\nu+1}(x)/K_{\nu}(x) \le 
B^{(K)}_{\nu}(+\infty,x_*,x)$,
\end{enumerate}
where the equality only takes place at $x=x_*$. The inequalities are valid for the values of
 $\nu$ specified earlier for each particular case ($\nu\ge 1/2$ in the worst case
 for $I_{\nu-1}(x)/I_{\nu}(x)$ and $\nu \ge 2$ in the worst case for $K_{\nu+1}(x)/K_{\nu}(x)$).
 
In addition, we have
$$
\begin{array}{l}
\displaystyle\lim_{x_*\rightarrow 0^+}B_{\nu}^{(I)}(0,x_*,x)=B_\nu^{(3,0)}(x),\,
\displaystyle\lim_{x_*\rightarrow +\infty}B_\nu^{(I)}(0,x_*,x)=B_\nu^{(1,2)}(x),\\
\displaystyle\lim_{x_*\rightarrow 0^+}B_\nu^{(I)}(+\infty,x_*,(x)=B_{\nu}^{(2,1)}(x),\,
\displaystyle\lim_{x_*\rightarrow +\infty}B_\nu^{(I)}(+\infty,x_*,x)=B_\nu^{(0,3)}(x)
\end{array}
$$
in a certain range of $\nu$ (at least $\nu\ge 1/2$) where $B^{(n,m)}_\nu(x)$ are the bounds in Table \ref{tablei}. The bounds $B_\nu^{(K)}$ satisfy the same relations
with respect to the bounds of Table \ref{tablek} (for $\nu\ge 2$ in the worst case).
\end{theorem}

\section*{Appendix}

Using \cite[10.4.1]{Olver:2010:BF} we obtain the following expansion as $x\rightarrow +\infty$:

\begin{equation}
\label{expin}
\Frac{I_{\nu-1}(x)}{I_\nu (x)}=1+\Frac{\nu-1/2}{x}+\Frac{\nu^2-1/4}{2x^2}+\Frac{\nu^2-1/4}{2x^3}+{\cal O}(x^{-4})
\end{equation}
and from \cite[10.4.1]{Olver:2010:BF}
\begin{equation}
\label{expink}
\Frac{K_{\nu+1}(x)}{K_\nu (x)}=1+\Frac{\nu+1/2}{x}+\Frac{\nu^2-1/4}{2x^2}-\Frac{\nu^2-1/4}{2x^3}+{\cal O}(x^{-4})
\end{equation}

For $\nu\ge 0$ 
the series for the regular solution at $x=0$ \cite[10.25.2]{Olver:2010:BF} gives 

\begin{equation}
\label{serin}
\Frac{I_{\nu-1}(x)}{I_\nu (x)}=\Frac{2\nu}{x}+\Frac{x}{2(\nu+1)}-\Frac{x^3}{8(\nu+1)^2 (\nu+2)}+{\cal O}(x^5).
\end{equation}

For the modified Bessel function of the second kind as $x\rightarrow 0^+$, because 
\begin{equation}
\label{defk}
K_{\nu}(x)=\Frac{\pi}{2}\Frac{I_{-\nu}(x)-I_{\nu}(x)}{\sin(\nu\pi)},
\end{equation}
we have that for $\nu>0$, $\nu\notin {\mathbb N}$
\begin{equation}
\label{serk1}
\Frac{K_{\nu+1}(x)}{K_{\nu} (x)}=-\Frac{I_{-\nu-1}(x)}{I_{-\nu}(x)}(1+{\cal O}(x^{2m})),\,m=\min\{2,2\nu\}
\end{equation}
and the expansion (\ref{serin}) can be used with $\nu$ replaced by $-\nu$. In particular, 
\begin{equation}
\Frac{K_{\nu+1}(x)}{K_{\nu} (x)}=\Frac{2\nu}{x}+\Frac{x}{2(\nu-1)}+{\cal O}(x^3),\,\nu>1 ,
\end{equation}

\begin{equation}
\Frac{K_{\nu+1}(x)}{K_{\nu} (x)}
=\left(\Frac{2\nu}{x}+\Frac{x}{2(\nu-1)}\right)
\left(1+\Frac{\Gamma (1-\nu)}{\Gamma (\nu+1)} \left(\Frac{x}{2}\right)^{2\nu}+{\cal O}(x^{4\nu})\right),\,\nu\in (0,1)
\end{equation}
and
\begin{equation}
\Frac{K_{\nu+1}(x)}{K_\nu (x)}={\cal O}(x^{-2\nu-1})\,\nu\in (-1,0).
\end{equation}

Finally, for $\nu=n\in {\mathbb N}$, 
the first $n$ terms in the expansion of $K_{\nu+1}(x)/K_{\nu}(x)$ are obtained from (\ref{serk1}),
using the first $n$ terms in (\ref{serin}) and adding a logarithmic factor to the error term in (\ref{serk1}).

\section*{Acknowledgements}
The author acknowledges support from Ministerio de Ciencia e Innovaci\'on, projects PGC2018-098279-B-I00 (MCIU/AEI/FEDER, UE) and
 PID2021-127252NB-I00 (MCIN/AEI/10.13039/501100011033/FEDER, UE).

%\section*{References}

\bibliography{bestn}

\end{document}